\documentclass[preprint]{elsarticle}
\usepackage{lineno,hyperref}
\modulolinenumbers[5]

\usepackage{url}

\usepackage{amsmath,amsthm,amsfonts,amssymb}
\usepackage[ruled,vlined]{algorithm2e}
\usepackage{xcolor}
\usepackage{enumerate}
\usepackage{csquotes}

\usepackage{graphicx}
\graphicspath{ {Figures/} }
\usepackage{subcaption}

\usepackage{multicol}
\usepackage{multirow}

\usepackage{calc}
\usepackage{tikz}
\usetikzlibrary{shapes, decorations.markings, calc}

\newtheorem{Theorem}{Theorem}
\newtheorem{Proposition}[Theorem]{Proposition}

\newtheorem{Lemma}[Theorem]{Lemma}
\newdefinition{Definition}{Definition}
\newdefinition{Assumption}{Assumption}
\newdefinition{Remark}{Remark}
\newdefinition{Example}{Example}

\def \ba{\begin{array}}
\def \ea{\end{array}}

\def \bea{\begin{eqnarray}}
\def \eea{\end{eqnarray}}

\def \be{\begin{equation}}
\def \ee{\end{equation}}

\def \BEA{\begin{eqnarray*}}
\def \EEA{\end{eqnarray*}}

\def \BE{\begin{equation*}}
\def \EE{\end{equation*}}

\def \+{\dagger}

\def \bb{\mathbb}

\def \mc{\mathcal}

\def \disp{\displaystyle}

\def \tcol{\textcolor}

\def \S{\mathtt{S}}
\def \I{\mathtt{I}}
\def \D{\mathtt{D}}
\def \R{\mathtt{R}}

\def \U{\mathtt{U}}
\def \T{\mathtt{T}}

\definecolor{blue1}{RGB}{128, 191, 255}
\definecolor{redd1}{RGB}{255, 153, 128}
\definecolor{yelw1}{RGB}{255, 179, 26}
\definecolor{prpl1}{RGB}{204, 153, 255}
\definecolor{gren1}{RGB}{153, 230, 153}
\definecolor{blue2}{RGB}{51   153  255}
\definecolor{redd2}{RGB}{255  71   26}
\definecolor{yelw2}{RGB}{255  204  0}
\definecolor{prpl2}{RGB}{153  51   255}
\definecolor{gren2}{RGB}{51   204  51}
\definecolor{mygreen}{rgb}{0,0.7,0}

\def \ol{\overline}

\begin{document}

\begin{frontmatter}
\title{Modeling and Control of Epidemics through Testing Policies\tnoteref{titlenote}}
\tnotetext[titlenote]{This work is supported by European Research Council (ERC) under the European Union’s Horizon 2020 research and innovation programme (ERCAdG no. 694209, Scale-FreeBack, website: \url{http://scale-freeback.eu/}) and by Inria, France, in the framework of the Inria's Mission COVID-19.}

\author[mainaddress1]{Muhammad Umar B. Niazi\corref{correspondingauthor}}
\ead{muhammad-umar-b.niazi@inria.fr}
\cortext[correspondingauthor]{Corresponding author}

\author[mainaddress1]{Alain Kibangou}
\ead{alain.kibangou@univ-grenoble-alpes.fr}

\author[mainaddress1]{Carlos Canudas-de-Wit}
\ead{carlos.canudas-de-wit@gipsa-lab.fr}

\author[mainaddress1]{Denis Nikitin}
\ead{denis.nikitin@gipsa-lab.fr}

\author[mainaddress1]{Liudmila Tumash}
\ead{liudmila.tumash@gipsa-lab.fr}

\author[mainaddress2]{Pierre-Alexandre Bliman}
\ead{pierre-alexandre.bliman@inria.fr}

\address[mainaddress1]{Univ. Grenoble Alpes, CNRS, Inria, Grenoble INP, GIPSA-Lab, 38000 Grenoble, France.}
\address[mainaddress2]{Sorbonne Universit\'{e}, Universit\'{e} Paris-Diderot SPC, Inria, CNRS, Laboratoire Jacques-Louis Lions, \'{e}quipe Mamba, 75005 Paris, France.}

\begin{abstract}
Testing is a crucial control mechanism in the beginning phase of an epidemic when the vaccines are not yet available. It enables the public health authority to detect and isolate the infected cases from the population, thereby limiting the disease transmission to susceptible people. However, despite the significance of testing in epidemic control, the recent literature on the subject lacks a control-theoretic perspective. In this paper, an epidemic model is proposed that incorporates the testing rate as a control input and differentiates the undetected infected from the detected infected cases, who are assumed to be removed from the disease spreading process in the population. After estimating the model on the data corresponding to the beginning phase of COVID-19 in France, two testing policies are proposed: the so-called best-effort strategy for testing (BEST) and constant optimal strategy for testing (COST). The BEST policy is a suppression strategy that provides a minimum testing rate that stops the growth of the epidemic when implemented. The COST policy, on the other hand, is a mitigation strategy that provides an optimal value of testing rate minimizing the peak value of the infected population when the total stockpile of tests is limited. Both testing policies are evaluated by their impact on the number of active intensive care unit (ICU) cases and the cumulative number of deaths for the COVID-19 case of France.
\end{abstract}


\end{frontmatter}

\section{Introduction}

The history of humanity is enameled with various pandemics whose consequences have had a durable impact on our societies. In almost all cases, vaccination is presented as a panacea, however, before vaccination becomes possible, an initial response, as effective as possible, must be provided. This is what the COVID-19 epidemic, started in Wuhan, China, at the end of 2019, has taught us. COVID-19 was declared to be a pandemic by the World Health Organization (WHO) on March 11, 2020. The most common symptoms of the disease include fever, cough, fatigue, shortness of breath, and loss of the senses of smell and taste, where complications may include pneumonia and respiratory distress known as a severe acute respiratory syndrome (SARS). For more than a year, the primary mode of treatment had been symptomatic and supportive therapy \cite{cao2020,baden2020}, and no approved vaccine or specific antiviral treatment was available.

The pandemic shook the economy of the whole world with a significant reduction of exports, a decline in tourism, mass unemployment, and business closures \cite{loayza2020}. 
Governments and health authorities worldwide responded by implementing non-pharmaceutical intervention (NPI) policies such as travel restrictions, lockdown measures, social distancing, workplace hazard controls, closure of schools and workplaces, curfew strategies, and cancellation of public events. Many countries also upgraded existing infrastructure and personnel to increase testing capabilities and facilities for focused isolation. The public was instructed to wash hands several times a day, cover mouth and nose when coughing or sneezing, maintain a certain physical distance from other people, wear a face mask in public places/gatherings, and to monitor and self-isolate if the disease symptoms appear.
The extent to which such policies and measures have been implemented is called the stringency index of a country's response to the epidemic \cite{hale2020,hale2020b}. Each country responded in its capacity to find a suitable balance between saving lives and saving livelihoods, which \cite{glover2020} termed as a problem of health versus wealth.
Livelihoods can be saved through the implementation of suitable relief and recovery measures for people and small businesses. On the other hand, lives can be saved through the implementation of testing policies and NPIs. In other words, there is a direct relation between the stringency index of the government and `saving lives'.

All the above strategies and policies are considered to be the control mechanisms for the epidemic. Such strategies fall under two categories: mitigation and suppression \cite{ferguson2020,walker2020}. The mitigation strategies slow down the rate of transmission of disease or `flatten the curve'. However, they do not necessarily stop the spreading of the disease, which is the goal of suppression strategies. Given the required objectives (for e.g., minimizing the number of deaths caused by the epidemic) and constraints (for e.g., socio-economic costs), the problem of finding optimal strategies for epidemic control has been recently studied under the framework of optimal control theory.

\subsection{Literature review}
To understand, predict, and control the evolution of the COVID-19 epidemic, a huge effort has been devoted by the researchers to design models as accurate and as effective as possible. Each model, by and large, is a variant and/or an extension of SIR (susceptible, infected, recovered) and SEIR (susceptible, exposed, infected, recovered) models, which describe the flow of population through three or four mutually exclusive stages of infection, respectively (see \cite{kermack1927} and \cite{hethcote2000} for a comprehensive review). These basic models have few parameters that are easy to identify \cite{massonis2020}, and are considered as population models that view the epidemic from the macroscopic perspective. This is in contrast with the approaches that capture heterogeneity of population structure such as network epidemic models \cite{khanafer2016,pare2018,pare2020} or metapopulation epidemic models \cite{colizza2008,pastor2015,della2020}, that view the epidemic from the microscopic perspective. In what follows, however, we study the epidemic through the macroscopic perspective of population models.

Following the outbreak of COVID-19, there has been an effort to produce comprehensive population models with a focus on different facets of the epidemic. Such models are more complex than simple SIR and SEIR, and include several intermediate stages that accurately portray the dynamics of the epidemic. For instance, \cite{lin2020} develops an extension of the SEIR model that incorporates the governmental actions (e.g., preventive measures and restrictions) and the individual behavioral reactions, whereas \cite{anastassopoulou2020} develops an extension of the SIR model that incorporates the number of deaths due to the epidemic. Another quite interesting model is the one developed in \cite{giordano2020} that considers an eight-compartment model called SIDARTHE, which includes eight stages of infection: susceptible (S), infected (I), diagnosed (D), ailing (A), recognized (R), threatened (T), healed (H), and extinct (E). A distinguishing feature of this model is that it differentiates between the infected individuals based on the severity of their symptoms and whether they are diagnosed by a health authority. It is crucial, as also emphasized in \cite{liu2020,ducrot2020}, to differentiate between diagnosed and undiagnosed individuals because the former are typically isolated and are less likely to spread the infection. Similar models have been adopted and extended to study optimal control policies for the epidemic such as the implementation of social distancing measures \cite{kohler2020,morato2020,perkins2020}, lockdown strategies \cite{casella2020,olivier2020,alvarez2020}, and heterogeneous policy responses based on age-groups \cite{acemoglu2020,brotherhood2020}.

In addition to the above NPI strategies, testing and isolating the infected population from the susceptible population is one of the most important strategies to control the epidemic spread. For instance, it has been reported that COVID-19 was eliminated from the Italian village Vo’Euganeo through testing both symptomatic and asymptomatic cases \cite{romagnani2020,day2020}. Moreover, on his media briefing\footnote{Website: \href{https://www.who.int/dg/speeches/detail/}{WHO Director-General Speech of March 16, 2020.} (Accessed 04/06/2020)} of March 16, 2020, Dr. Tedros Adhanom Ghebreyesus, the Director-General of WHO, urged the following: 

\begin{displayquote}
{\em ``Social distancing measures can help to reduce transmission and enable health systems to cope. Hand-washing and coughing into your elbow can reduce the risk for yourself and others. But on their own, they are not enough to extinguish this pandemic. It’s the combination that makes the difference. As I keep saying, all countries must take a comprehensive approach. But the most effective way to prevent infections and save lives is breaking the chains of transmission. And to do that, you must test and isolate. You cannot fight a fire blindfolded. And we cannot stop this pandemic if we don’t know who is infected. We have a simple message for all countries: TEST, TEST, TEST.''}
\end{displayquote}

COVID-19 can be detected through two types of tests known as type-1 (RT-PCR) and type-2 (serology). In the type-1 test, a swab is inserted into the subject's nose to qualitatively detect nucleic acid from SARS-CoV-2 in the upper and lower respiratory specimens \cite{FDA1}, which enables one to detect whether the subject is currently infected with COVID-19.
Type-2 test, on the other hand, is a serum test that detects relevant antibodies enabling one to know whether the subject was infected in the past with COVID-19. Both types of tests are important in the control of an epidemic. Type-1 tests help to limit the disease spread by the identification of infected individuals and their contact tracing \cite{walque2020}. Type-2 tests, on the other hand, are useful in reducing the size of the testable population for type-1 tests \cite{winter2020} that helps to increase the testing specificity. However, the type-1 test was considered to be the only recommended method for the identification and laboratory confirmation of COVID-19 cases according to the WHO \cite{WHO1}. Moreover, only type-1 tests can provide information in real-time related to describe the outburst of the epidemic, which is the reason that the datasets related to testing only include type-1 tests\footnote{Website: \href{https://ourworldindata.org/coronavirus-testing\#different-types-of-tests-for-covid-19}{Our World in Data: Coronavirus (COVID-19) Testing}. (Accessed 30/09/2020)}.

Following the recommendation of the WHO director, with different levels of setups, many countries increased their testing capacities, while others feared the economic burden of intensive testing policy. However, \cite{eichenbaum2020,salathe2020} show that such a burden is only short-term and, on the contrary, intensive testing reduces the `overall' cost of the epidemic in the long run because it enables the government to gain rapid control of the epidemic and revive the economy of a country.
Testing enables the health authority to identify and isolate the infected people from the susceptible population, which limits the transmission of the disease. Therefore, testing is considered to be a crucial control mechanism for the epidemic \cite{chowell2003}. However, few attempts have been dedicated to study the testing policy for an epidemic from a control-theoretic perspective. 

In somewhat similar to a resource allocation problem \cite{nowzari2016,nowzari2017} in epidemic control, \cite{pezzutto2020} poses the optimal test allocation as a well-known sensor selection problem in control theory, whereas \cite{ely2020} poses it as a welfare maximization problem by considering specificity and sensitivity of tests. The main assumption in these papers, however, is the availability of information portfolios of all individuals in a society, which enables the decision makers to compute the infection probability of individuals and utility loss for each individual in case of decision errors. On the other hand, \cite{piguillem2020,berger2020} study the problem of testing policy from an economic perspective, where the goal is to find an optimal testing policy that minimizes the total number of quarantined people to incur minimal cost on the economic activity of a country while also mitigating the epidemic spread. Without such a policy, health authorities usually resort to indiscriminate quarantining of people that burdens the economy of a country without any reason. Therefore, testing allows to identify and isolate the positive cases for case-dependent quarantining. Another aspect of testing policy is studied in \cite{charpentier2020}, which computes an optimal trade-off between testing effort and lockdown intervention under the constraint of limited Intensive Care Units (ICU).

\subsection{Our contribution}
At the onset of an epidemic, the effective treatments or vaccines are usually not available. In such a case, it is important to devise effective testing policies for epidemic control.
We introduce a modified SIR model named SIDUR --- susceptible (S), undiagnosed infected (I), diagnosed infected (D), unidentified recovered (U), and identified removed (R) --- to study epidemic control through testing. Similar to \cite{giordano2020, liu2020, ducrot2020}, we differentiate between the undiagnosed and diagnosed infected population. We assume that the diagnosed infected population are either quarantined and/or hospitalised and only the undiagnosed infected population is responsible for the disease transmission to the susceptible population. The identified removed population consists of people who recover or die after being diagnosed and the unidentified recovered population consists of people who recover without getting diagnosed. The control input in the SIDUR model is the testing rate defined as the number of tests performed per day, where the influence of the control is directly linked with the testing specificity. The testing specificity determines the probability of detecting an undiagnosed infected person through a test, which, for instance, can be increased through efficient contact tracing.

First, we estimate the parameters of SIDUR model using the COVID-19 data of France. Then, we propose two testing policies for epidemic control: 1) Best effort strategy for testing (BEST) and 2) Constant optimal strategy for testing (COST).

The best-effort strategy for testing (BEST) is a suppression strategy for an epidemic that provides the minimum number of tests to be performed per day in order to stop the epidemic spread. Thus, BEST is meaningful only during the spreading phase of the disease. We provide an algorithm to compute the number of tests required by BEST policy. Since BEST is a suppression strategy that stops the epidemic growth immediately, it usually requires a large number of tests to be performed per day. However, if implemented earlier, BEST requires feasible number of tests. We illustrate this for the COVID-19 case of France by plotting the number of tests required by BEST with respect to time.

The constant optimal strategy for testing (COST) is a mitigation strategy when the total stockpile of tests is limited. It provides the optimal number of tests per day that must be allocated in a certain time interval from the onset of the epidemic in order to minimize the peak of infected population. When the stockpile finishes at the terminal time, then the tests are not performed anymore. The main idea of COST is the consideration of two peaks of the epidemic; one that occurs before the stockpile of tests finishes and one that occurs after the stockpile finishes. If the stockpile of tests is allocated constantly such that it reduces the first peak, it will however result in the increase of second peak, and vice versa. Thus, the optimal allocation of tests per day is the one that minimizes both peaks of the infected population, which occurs when both peaks are equal. Both BEST and COST policies are compared with the actual COVID-19 testing scenario of France through the prediction of the number of active ICU cases and the cumulative number of deaths.

\subsection{Paper organization}
In Section~\ref{sec:model}, we describe the model, inputs, and outputs, and compute the basic and effective reproduction numbers of the model. Section~\ref{sec:data} illustrates the French COVID-19 data and provides the data imputations to infer the missing data from the raw data. Then, Section~\ref{sec:estimation} provides the estimation and fitting of model for the French COVID-19 case. Finally, in Section~\ref{sec:control}, we propose two testing policies BEST and COST, and evaluate them by comparing the predicted number of active ICU cases and the predicted cumulative number of deaths with the actual data.

\section{SIDUR model with testing policy} \label{sec:model}

We consider a five-compartment model with the purpose of evaluating and devising testing policies. We assume that testing allows for diagnosing and isolating the infected people from the population to prevent the transmission of the disease to the susceptible population. The proposed model is named SIDUR, which corresponds to the five compartments: susceptible (S), undiagnosed infected (I), diagnosed infected (D), unidentified recovered (U), and identified removed (R). The model is characterized by four parameters and one input, which is the testing rate.

\subsection{Model design}

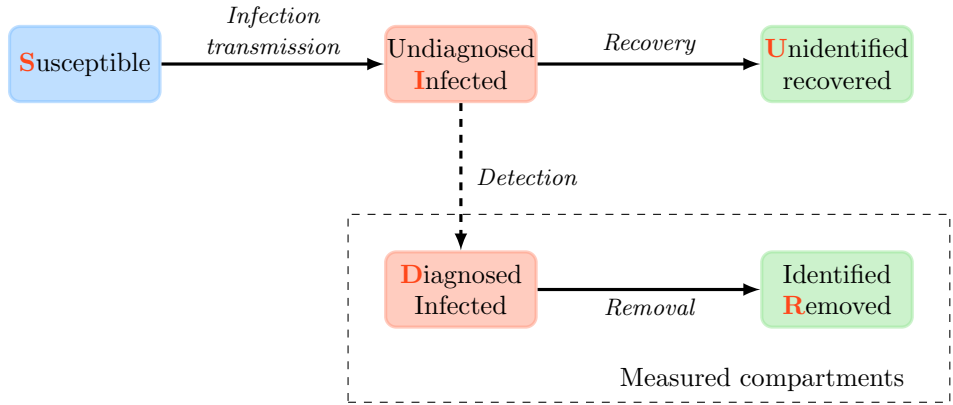
\begin{figure}[!htb]
\begin{center}
\begin{tikzpicture}

    
    \draw[dashed] (3.5,-1) rectangle (11.5,1.5);
    \node at (9,-0.7) {Measured compartments};
    
	\draw[blue1, fill=blue1, very thick, rounded corners, opacity = 0.5] (-1,3) rectangle (1,4);
	\node at (0,3.5) {\tcol{redd2}{\bf S}usceptible};

	\draw[-latex, very thick] (1.02,3.5) -- (3.98,3.5);
	\node[text width=2.5cm, anchor=south, align=center] at (2.5,3.5) {\small \it Infection transmission};
	
	\draw[redd1, fill=redd1, very thick, rounded corners, opacity = 0.5] (4,3) rectangle (6,4);
	\node[text width=2cm, align=center] at (5,3.5) {Undiagnosed \tcol{redd2}{\bf I}nfected};
	
	
	\draw[-latex, dashed, very thick] (5,2.98) -- (5,1.02);
	\node[text width=2.5cm, anchor=west, align=center] at (4.5,2) {\small \it Detection};
	
	\draw[redd1, fill=redd1, very thick, rounded corners, opacity = 0.5] (4,0) rectangle (6,1);
	\node[text width=2cm, align=center] at (5,0.5) {\tcol{redd2}{\bf D}iagnosed Infected};

	\draw[-latex, very thick] (6.02,3.5) -- (8.98,3.5);
	\node[text width=2.5cm, anchor=south, align=center] at (7.5,3.5) {\small \it Recovery};
	
	\draw[gren1, fill=gren1, very thick, rounded corners, opacity = 0.5] (9,3) rectangle (11,4);
	\node[text width=2cm, align=center] at (10,3.5) {\tcol{redd2}{\bf U}nidentified recovered};
	
	\draw[-latex, very thick] (6.02,0.5) -- (8.98,0.5);
	\node[text width=2.5cm, anchor=north, align=center] at (7.5,0.5) {\small \it Removal};
	
	\draw[gren1, fill=gren1, very thick, rounded corners, opacity = 0.5] (9,0) rectangle (11,1);
	\node[text width=2cm, align=center] at (10,0.5) {Identified \tcol{redd2}{\bf R}emoved};
\end{tikzpicture}
\caption{Block diagram of SIDUR model.}
\label{fig:SIDUR_block}
\end{center}
\end{figure}

SIDUR is a compartmental model depicted in Figure~\ref{fig:SIDUR_block}.
At time $t\in\bb{R}_{\geq 0}$, each compartment is characterized by a single state:
\begin{itemize}
    \item $x_\S(t)$: Number of susceptible people who are prone to the disease. 
    \item $x_\I(t)$: Number of infected people who are undetected by the public health authority.
    \item $x_\D(t)$: Number of diagnosed people who are infected and detected by a test.
    \item $x_\U(t)$: Number of unidentified recovered people who recover without getting diagnosed.
    \item $x_\R(t)$: Number of removed people who either recover or die after being diagnosed.
\end{itemize}
The development of the model is based on the following three assumptions:

\begin{Assumption} \label{assump:constant_pop}
The population remains constant during the evolution of epidemic:
\[
x_\S(t) + x_\I(t) + x_\D(t) + x_\U(t) + x_\R(t) = N
\]
where $N$ stands for the total population. \hfill $\lrcorner$
\end{Assumption}

During the evolution of the epidemic, the change in population due to births, deaths due to causes exclusive to the epidemic, and inflow/outflow of travelers from/to other countries is assumed to be negligible with respect to the total population. In other words, during the time horizon under consideration, the birth rate is approximately equal to the death rate and the rate of inflow of travelers is approximately equal to the rate of outflow.

\begin{Assumption} \label{assump:undiagnosed_inf}
Only the undiagnosed infected population $x_\I(t)$ is responsible for the disease transmission to the susceptible population $x_\S(t)$. \hfill $\lrcorner$
\end{Assumption}

The diagnosed infected people $x_\D(t)$ are isolated from the population in the form of quarantine or hospitalization. This means that their contact with susceptible people is restricted or, for hospitalized cases, under strict sanitary measures. Therefore, disease transmission due to a diagnosed infected person is assumed to be unlikely. It is possible, however, that a diagnosed infected person could have transmitted the disease to other people before getting diagnosed.

\begin{Assumption} \label{assump:reported_deaths}
All the deaths from COVID-19 are identified and reported; they are included in the removed population $x_\R(t)$ along with the people who recover after being diagnosed. \hfill $\lrcorner$
\end{Assumption}

The non-surviving cases of the disease usually have severe symptoms. Therefore, they are assumed to be diagnosed and hospitalized before their death.

\begin{Assumption} \label{assump:immunity}
The acquired immunity of recovered population is sustainable enough. That is, the unidentified recovered population $x_\U(t)$ and the removed population $x_\R(t)$ cannot get infected again during the time interval considered in an epidemic. \hfill $\lrcorner$
\end{Assumption}

Once the infected cases recover, they develop antibodies that prevent them from getting infected again by the same variant of the disease. However, if the epidemic endures for an extended period of time, then other variants of the disease emerge causing the recovered population to become susceptible again. Nonetheless, we consider a beginning phase of the epidemic during which we assume that disease variants do not emerge and the acquired immunity of recovered population is sustainable enough.

Based on the above assumptions, the model is given as 
\begin{subequations}
	\label{eq401}
	\begin{eqnarray}
	\label{eq401a}
	\dot{x}_\S(t) &=& \disp - \beta \; x_\S(t) \frac{x_\I(t)}{N} \\ [0.25em]
	\label{eq401b}
	\dot{x}_\I(t) &=& \disp \beta \; x_\S(t) \frac{x_\I(t)}{N} - u(t) \frac{x_\I(t)}{x_\T(t)} -\gamma x_\I(t) \\ 
	\label{eq401c}
	\dot{x}_\D(t) &=& u(t) \frac{x_\I(t)}{x_\T(t)} - \rho x_\D(t) \\ [0.5em]
	\label{eq401d}
	\dot{x}_\U(t) &=&\gamma x_\I(t)\\ [1em]
	\label{eq401e}
	\dot{x}_\R(t) &=&\rho x_\D(t)
	\end{eqnarray}
\end{subequations}
where $\beta$, $\gamma$, and $\rho$ are the infection,  recovery, and removal rates, respectively, $u(t)$ is the testing rate,
\be \label{eq:testable}
\ba{ccl}
x_\T(t) &=& x_\I(t) + (1-\theta) \left(x_\S(t) + x_\U(t)\right) \\
&=& \theta x_\I(t) + (1-\theta) \left(N- x_\D(t) - x_\R(t)\right)
\ea
\ee
is the testable population, and $\theta$ is the testing specificity parameter which takes values in the interval $[0,1]$.

The recovery rate $\gamma$ is the inverse of the average recovery time $1/\gamma$ after which an undiagnosed infected person recovers, and the removal rate $\rho$ is the inverse of the average removal time $1/\rho$ after which a diagnosed infected person recovers or dies. The average recovery time is expected to be shorter than the average removal time, i.e., $\gamma \geq \rho$, because the undiagnosed infected population that comprises the undetected asymptomatic cases and cases with mild symptoms recover faster than the diagnosed population that comprises mostly the cases with severe symptoms.

The infection rate $\beta$ is the product of the frequency of contacts among the susceptible and infected populations and the probability of disease transmission after a contact has been made. Thus, the parameters $\beta$, $\gamma$, and $\rho$  are related to the disease biology. However, the value of $\beta$ can also be partially impacted by non-pharmaceutical interventions (NPI) such as social distancing, lockdown, confinement, travel restrictions, and preventive policies (i.e., to maintain a certain distance from other people, to wear a face mask in public spheres, to wash/sanitize hands more often, etc.). The value of $\beta$ is expected to be smaller when NPI's are implemented than the value of $\beta$ when no NPI is implemented. Depending on the time periods during which different NPI policies are implemented, we assume the infection rate $\beta$ to be piecewise constant.

The testing specificity parameter $\theta$, on the other hand, is solely dependent on the testing policy implemented by the public health authority. Given that the testing rate is constant, the value of $\theta$ will be larger when the tests are allocated efficiently through contact tracing than the value of $\theta$ when the tests are performed randomly. However, there are other factors that can also influence $\theta$, for example, if only the people with severe symptoms are tested, then the probability $x_\I/x_\T$ of detecting an infected person from the testable population is equal to one, i.e., the testing specificity parameter $\theta=1$. This is to indicate that the larger value of $\theta$ doesn't necessarily imply the efficiency of testing policy, rather it only signifies the specificity of tests. Depending on the time periods during which different testing policies are implemented, we assume the testing specificity parameter $\theta$ to be piecewise constant.

\subsection{Control input and testable population}

We consider the testing rate $u(t)$ to be the model input which depends on three factors: the daily testing capacity $c(t)$, the remaining stockpile of tests $r(t)$, and the testable population $x_\T(t)$. We consider $c(t)$ to be time-varying in order to take into account the fact that the capacity of testing can change on a daily basis. Moreover, on a given day $t$, one cannot do more tests than the daily testing capacity $c(t)$, the remaining stockpile of tests $r(t)$, or the testable population $x_\T(t)$. Therefore, we have
\be \label{eq:testing_rate}
u(t) := \min\left( c(t), r(t), x_\T(t) \right).
\ee
In case the total stockpile of tests $r_{\max}$ is limited, the remaining stockpile of tests at time $t$ is given by
\[
r(t) := r_{\max} - \int_0^t u(\eta) d\eta.
\]
In case new tests can be produced and supplied easily, \eqref{eq:testing_rate} is simply given by 
\[
u(t) := \min\left( c(t), x_\T(t) \right).
\]
Usually, the testable population $x_\T(t)$ is much larger than the daily testing capacity, thus
\be \label{eq:testing_rate2}
u(t) = \left\{\ba{ll}
c(t), & \text{if}~r_{\max}~\text{is unlimited} \\
\min(c(t),r(t)), & \text{if}~r_{\max}~\text{if limited}.
\ea\right.
\ee

In order to diagnose the infected people at time $t$, the tests are allocated to a proportion of the testable population $x_\T(t)$, which is a sample from the total population $N$. From \eqref{eq:testable}, it is obvious that the infected population $x_\I(t)\leq x_\T(t)$ at any given time $t$. Thus, given the testing specificity parameter $\theta\in[0,1]$, the probability of detecting an infected person per test in a homogeneous population structure is given by $x_\I(t)/x_\T(t)$. 

The testing specificity parameter $\theta$ allows for the adjustment of the testable population to accommodate for the detection rate of tests. In most countries, at the beginning of an epidemic outbreak, the number of available tests are limited. Thus, the available tests are usually utilized to confirm the symptomatic infected cases or to diagnose certain people such as medical care agents, politicians, athletes, etc. In such a case, the testable population is close to the infected population and the value of $\theta$ increases to approximately one. 
Once the capacity of testing is increased, the size of the testable population is also increased that can include, for example, contacts of diagnosed people, the whole population of a city where a cluster is identified, travellers, etc. As a consequence, the value of $\theta$ decreases. 

\subsection{Outflows from the model compartments}

The SIDUR model is described by the one-way transfer of population between compartments, where an outflow from one compartment is the inflow to the other compartments. Thus, it suffices to describe only the outflows from the compartments to describe the dynamics of the model.

\paragraph{Infection transmission}
In the beginning of the epidemic, most of the population is in the susceptible compartment (S) with the exclusion of those who are initially infected and/or diagnosed. Some of the susceptible people in S may get infected and leave this compartment when they come in contact with an infected person. The rate of the outflow from this compartment is according to the infection transmission rate, which depends on the product of the number of susceptible and infected populations, and is given as
\[
\beta x_\S(t) \frac{x_\I(t)}{N}
\]
where $\beta$ is the infection rate. The term $x_\I(t)/N$ is the proportion of undetected infected population at any time $t$ in a homogeneous population structure. Note that in light of Assumption~\ref{assump:undiagnosed_inf}, diagnosed population $x_\D(t)$ does not participate in the infection transmission because they are either quarantined and/or hospitalized, i.e., they are temporarily removed from the population. Finally, by Assumption~\ref{assump:immunity}, there is no inflow to the susceptible compartment.

\paragraph{Detection}
The outflow from the infected compartment (I) is either due to detection (i.e., transfer to the diagnosed compartment (D)) or recovery without detection (i.e., transfer to the unidentified recovered compartment (U)). The first outflow is due to the testing rate $u(t)$, i.e., the number of tests performed per day. Since the probability of detecting an infected person from a testable population by a single test is $x_\I(t) / x_\T(t)$, therefore we have
\[
u(t) \frac{x_\I(t)}{x_\T(t)}
\]
the rate of diagnosing the infected population in I compartment.

\paragraph{Recovery}
The second outflow from the I compartment consists of those people who are not diagnosed and recover naturally with an average recovery period of $1/\gamma$. The unidentified recovered compartment (U) accumulates the infected people who recover naturally without being detected with a recovery rate $\gamma$.

\paragraph{Removal}
The diagnosed compartment (D) admits $u(t) x_\I(t)/x_\T(t)$ as an inflow, whereas the outflow is $\rho x_\D(t)$ with $\rho$ being the removal rate. That is, $1/\rho$ is the average time period after which a typical diagnosed person either recovers or dies. The removed compartment (R) accumulates the diagnosed people who die or recover with a removal rate $\rho$.

\subsection{Output signals from the model}
The outputs signals $y_i(t)$, $t\in\bb{R}_{\geq 0}$, $i={1,2,\dots,m}$, from the model (or model outputs) correspond to the sampled output measurements $\ol{y}_i(k)$, $k\in\bb{Z}_{\geq 0}$, from the data (or data outputs), which can be approximated by a continuous signals $\ol{y}_i(t)$, respectively. The data outputs and model outputs are related as follows:
\[
\ol{y}_i(t) = y_i(t) + w_i(t)
\]
where $w_i(t)$ is the measurement noise/error. 

The model outputs are functions of the states of SIDUR model, which are of two types: (i) outputs whose functions are known and (ii) outputs whose functions are not known.
First, we define three model outputs whose function is known. 
\begin{itemize}
    \item Cumulative number of diagnosed people
    \begin{equation}\label{total_infected_cases}
        y_1(t)= x_\D(t) + x_\R(t).
    \end{equation}
    \item Cumulative number of removed people
    \begin{equation}\label{total_recovered_cases}
        y_2(t)= x_\R(t).
    \end{equation}
    \item Number of positively tested people (or positive test results) per day
    \begin{equation}\label{testedpos_cases}
        y_3(t)= u(t) \frac{x_\I(t)}{x_\T(t)}.
    \end{equation}
\end{itemize}
These model outputs are fitted with the data outputs in order to estimate the model parameters $\beta,\theta,\gamma,\rho$ in Section~\ref{sec:estimation}. Note that these model outputs are related to each other. Since the number of diagnosed infected people at any time $t$ can be obtained as $x_\D(t) = y_1(t) - y_2(t)$, which is also known as the number of active diagnosed cases, we obtain the following relation between $y_1(t)$ and $y_2(t)$ from \eqref{eq401e}
\be \label{y2y1_relation}
\dot{y}_2(t) = \rho (y_1(t) - y_2(t)).
\ee
On the other hand, the number of positive test results per day $y_3(t)$ is related to the cumulative number of diagnosed cases $y_1(t)$ by the following relation
\be \label{output_relation}
y_3(t) = \dot{x}_\D(t) + \dot{x}_\R(t) = \dot{y}_1(t).
\ee
The cumulative number of diagnosed people $y_1(t)$ can be obtained by integrating the daily number of positive test results as
\be \label{output_relation2}
y_1(t)-y_1(0) = \int_0^t y_3(\eta) d\eta.
\ee
These output relations \eqref{y2y1_relation}, \eqref{output_relation}, \eqref{output_relation2} are used to infer the missing data from the available data in Section~\ref{sec:data}. 

Second, one can note that the number of active ICU patients, denoted $B(t)$, (or ICU beds occupied) and the cumulative number of deaths, denoted $E(t)$, (or extinct cases) are positively correlated with the total number of active infected cases $A(t)=x_\I(t) + x_\D(t)$ and the cumulative number of infected cases $I(t)=N-x_\S(t)$, respectively. Therefore, we define two additional model outputs:
\begin{itemize}
    \item Number of active Intensive Care Unit (ICU) cases (or the number of ICU beds currently occupied by COVID-19 patients):
    \be \label{activeICU_cases}
    y_4(t) := B(t) = g(A(t-\psi^{-1}))
    \ee
    where $A(t) = x_\I(t)+x_\D(t)$ is the number of active infected cases, $\psi^{-1}$ is the average time period a typical COVID-19 ICU case takes from getting infected to being admitted to ICU, and the function $g$ is to be defined.
    \item Cumulative number of deaths due to COVID-19 (or extinct cases):
    \be \label{total_dead_cases}
    y_5(t) := E(t) = h(I(t-\phi^{-1}))
    \ee
    where $I=N-x_\S(t)$ is the cumulative number of infected cases, $\phi^{-1}$ is the average time period a typical COVID-19 extinct case takes from getting infected to death, and the function $h$ is to be defined.
\end{itemize}
The functions $g$ and $h$ will be estimated from available data on the number of active ICU cases and the cumulative number of deaths, respectively, in Section~\ref{sec:estimation}. Then, the model outputs $B(t)$ and $E(t)$ will be used as performance outputs to evaluate the testing strategies proposed in Section~\ref{sec:control}.

\subsection{Basic and effective reproduction numbers}

An important quantity to assess the epidemic potential of a disease is the {\em basic reproduction number} $R_0$, which is defined as the expected number of secondary infected cases produced by a single infected person in a completely susceptible population \cite{hethcote2000}. If $R_0>1$, then each generation of infected cases produces more secondary cases in the next generation and the disease has a potential of becoming an epidemic. If $R_0<1$, then each generation of infected cases produces less secondary cases in the next generation and the disease will eventually die out. It is worth noticing, however, that the definition of $R_0$ assumes that the people around a primary infected case are all susceptible. This suggests that determining $R_0$ is important only at the onset of an epidemic. However, in the later stages, more people get infected and not all people around an infected person are necessarily susceptible. As more people get infected, the conditions favoring the disease to propagate change and the number of susceptible people that an infected person infects is actually less than that what $R_0$ predicts. Thus, a more suitable quantity during the later stages of the epidemic is the {\em effective reproduction number} $R_t$, which takes into account the proportion of susceptible people in the total population \cite{rothman2008}.

For the SIDUR model, we define the effective reproduction number $R_t$ to be the ratio of the inflow and the outflow of the undiagnosed infected compartment (I). That is, $R_t$ is the ratio of the number of newly infected people and the number of newly diagnosed and recovered people at time $t$. If $R_t<1$, this means that more people are being diagnosed and recovered than the people being infected at time $t$, which implies that $x_\I(t)$ will decrease. If $R_t>1$, this means that more people are being infected than the people being diagnosed and recovered at time $t$, which implies that $x_\I(t)$ will increase. Notice that the definition of the basic reproduction number $R_0$ is same as $R_t$ for $t=0$.

To derive the expression of $R_t$, we consider the model equation \eqref{eq401b}, where the undiagnosed infected population $x_\I(t)$ satisfies \[
\dot{x}_\I(t) = \beta x_\S(t) \frac{x_\I(t)}{N} - u(t) \frac{x_\I(t)}{x_\T(t)} - \gamma x_\I(t).
\]
The positive rate (inflow) $\beta x_\S(t) x_\I(t) / N$ tells how many new infections will be generated in the next moment, and the negative rates (outflows) $u(t) x_\I(t)/x_\T(t)$ and $\gamma x_\I(t)$ tell how many infected people will be diagnosed or recovered in the next moment, respectively. Therefore, the effective reproduction number is the following ratio
\be \label{eq:eff_R}
R_t = \frac{\beta x_\S(t) \frac{x_\I(t)}{N}}{u(t) \frac{x_\I(t)}{x_\T(t)} + \gamma x_\I(t)} = \frac{\beta}{\frac{u(t)}{x_\T(t)} + \gamma} \frac{x_\S(t)}{N}.
\ee

To derive the expression of $R_0$, we consider the expression of $R_t$ at $t=0$, which corresponds to the onset of the epidemic. For $t=0$, we can assume few infected cases, which implies that $x_\S(0)\approx N$ and $x_\T(0)\approx (1-\theta)N$. Under these approximations, we have
\be \label{eq:basic_R0}
R_0 = \frac{\beta}{\frac{u(0)}{(1-\theta) N} + \gamma}.
\ee
This expression can also be obtained by following the methodology of \cite{driessche2008}.
Notice that the reproduction number $R_0$ depends on the initial testing policy $u(0)$. This indicates that it is possible to suppress the epidemic in the beginning by having an intensive testing policy, which can be seen, for example, in the case of South Korea \cite{oh2020}. In general, however, we have $u(0)\approx 0$, which gives $R_0\approx \beta / \gamma$.

\section{Data acquisition and imputation} \label{sec:data}

The data related to COVID-19 in France is collected from the French government's platform for publicly available data\footnote{Website: \href{https://www.data.gouv.fr/fr/datasets/}{Open platform for French public data.} (Accessed 17/10/2020)} for the time period of January 24 to July 01, 2020. In particular, we use datasets provided by the French Ministry of Social Affairs and Health (Minist\`{e}re des Solidarit\'{e}s et de la Sant\'{e} (MSS)) and the French Public Health Agency (Sant\'{e} Publique France (SPF)). From MSS, we obtain the data about different categories of people affected by \mbox{COVID-19}, i.e., diagnosed, hospitalized, recovered from hospitals, and dead. From SPF, we obtain the data for the number of PCR tests performed and positive test results obtained per day.

The data obtained from both sources is incomplete in several aspects. For instance, the data for the number of recovered people does not record those who recover from their homes after being diagnosed. These people do not show severe symptoms of the disease and, therefore, are not hospitalized, but are quarantined in their homes for some days. Only those who are hospitalized after being diagnosed are recorded as recovered when they are discharged from the hospital. On the other hand, the data for COVID-19 PCR tests is also incomplete. To illustrate this, we consider three intervals of time: (1)~January 24 to March 09, 2020, (2)~March 10 to May 12, 2020, and (3)~May 13 to July 01, 2020. There is no data available for the tests during the first interval. During the second interval, the testing data is collected only from the medical laboratories and not from the hospitals. However, we have reliable data only during the third interval which is collected both from the medical laboratories and the hospitals. 
Therefore, the data obtained from the above sources can be considered as a raw data which needs to be imputed.  


\subsection{Raw data} 

This subsection illustrates the data obtained from MSS and SPF without any modification.

\paragraph{Cumulative number of diagnosed cases}
We denote the data for the cumulative number of diagnosed cases by $\ol{y}_{1}$, which is illustrated in Figure~\ref{fig:raw_diagnosed} and corresponds to the model output $y_1(t)$ in \eqref{total_infected_cases}. It is also known as the total ``confirmed'' cases. This is a cumulative data for all the cases diagnosed with the disease through RT-PCR tests\footnote{Website: \href{https://www.santepubliquefrance.fr/content/download/228073/file/COVID-19_definition_cas_20200403.pdf}{Definition of a COVID-19 confirmed case by SPF.} (Accessed 14/07/2020)}. Thus, it includes both the active cases (those who are either admitted to the hospitals and/or quarantined) and the inactive cases (those who either recovered or died after being diagnosed). That is, $\ol{y}_{1}(k)$ corresponds to the sum of people in the diagnosed (D) and removed (R) compartments of the SIDUR model (Figure~\ref{fig:SIDUR_block}) on a given day $k$, as given in \eqref{total_infected_cases}. 

\begin{figure}[!htb]
    \centering
    \includegraphics[width=0.9\textwidth]{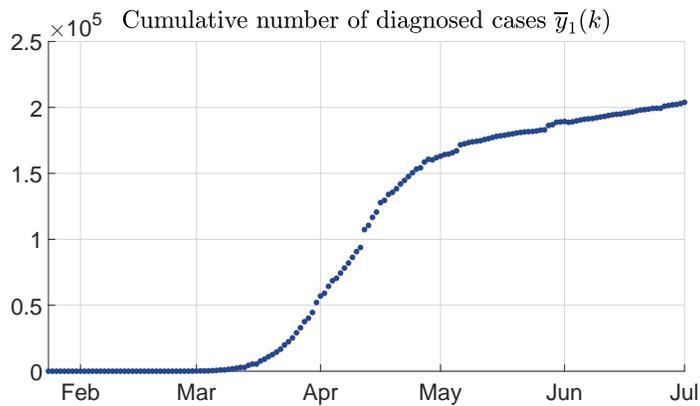}
	\caption{Cumulative number of diagnosed cases $\ol{y}_1(k)$ from January 24 to July 01, 2020. Source: MSS.}
\label{fig:raw_diagnosed}
\end{figure}

There is also an additional data for the diagnosed cases from French retirement homes (EHPAD). However, the French government database\footnote{Website: \href{https://dashboard.covid19.data.gouv.fr/vue-d-ensemble?location=FRA}{COVID-19 - France.} (Accessed 01/10/2020)} and several other international databases\footnote{Website: \href{https://www.ecdc.europa.eu/en/cases-2019-ncov-eueea}{European Centre for Disease Prevention and Control.} (Accessed 01/10/2020)} \footnote{Website: \href{https://www.worldometers.info/coronavirus/country/france/}{Worldometers/coronavirus/France.} (Accessed 01/10/2020)} do not add the diagnosed cases from EHPAD to the cumulative number of diagnosed (confirmed) cases. That is, the data for the cumulative number of diagnosed cases is considered to be inclusive of the diagnosed cases from EHPAD. However, in all the above databases, the data on cumulative number of deaths is collected separately from both the hospitals and EHPAD.

\paragraph{Number of active hospitalized and ICU cases}
The data on the number of active hospitalized cases is denoted as $\ol{H}$ and is illustrated in Figure~\ref{fig:data_hospitals} along with the number of active ICU cases $\ol{B}$. This data corresponds to the number of people who are admitted to the hospitals and/or ICU on a given day. That is, it is not a cumulative data. Moreover, it doesn't include those who were diagnosed but not hospitalized. That is, this data corresponds to a certain proportion of people in the diagnosed compartment (D) of the SIDUR model. This data is available from March 17, 2020, onward.

\begin{figure}[!htb]
    \centering
    \includegraphics[width=0.9\textwidth]{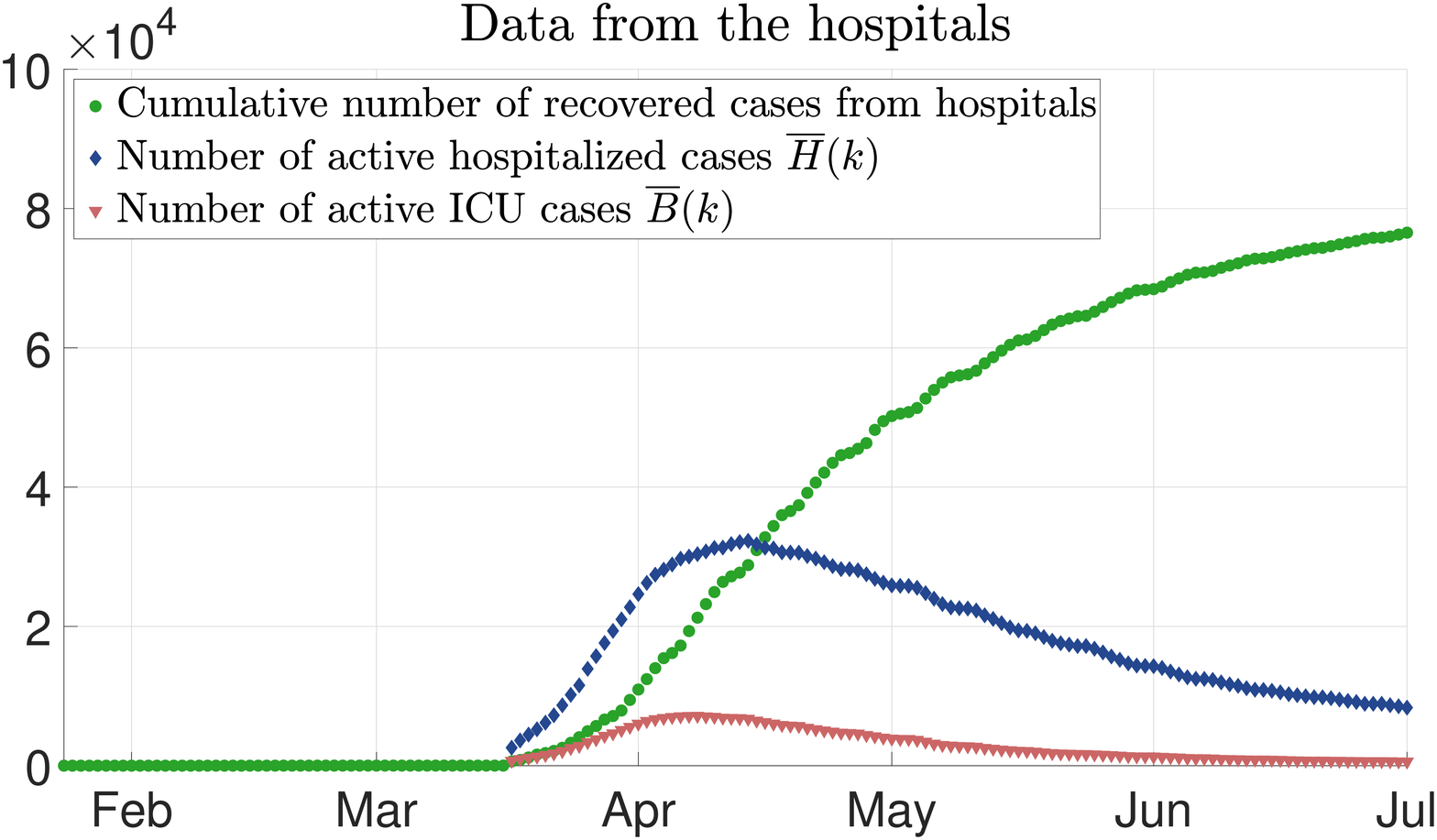}
	\caption{Total number of recovered cases who returned home after hospitalization from January 24 to July 01, 2020. The number of active COVID-19 hospitalized cases $\ol{H}(k)$ and ICU cases $\ol{B}(k)$ from March 17 to July 01, 2020. Source: MSS.}
\label{fig:data_hospitals}
\end{figure}

\begin{figure}[!htb]
    \centering
    \includegraphics[width=0.9\textwidth]{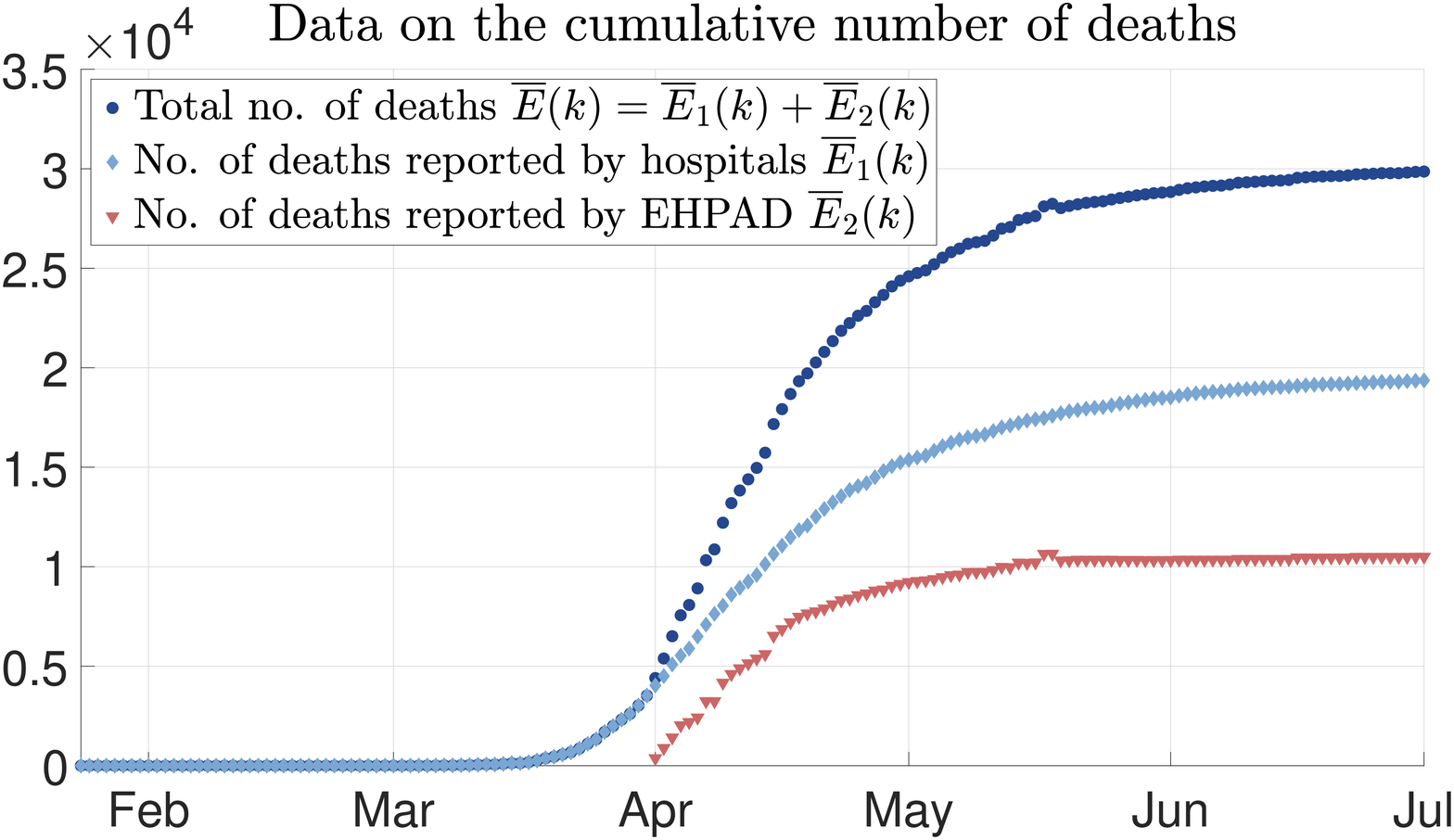}
	\caption{Total number of deaths from COVID-19 reported (a) by hospitals from January 24 to July 01, 2020, and (b) by retirement homes (EHPAD) from April 01 to July 01, 2020. Source: MSS.}
\label{fig:deaths}
\end{figure}

\paragraph{Cumulative number of recovered cases from hospitals}
This data is illustrated in Figure~\ref{fig:data_hospitals}. It corresponds to people who, after recovering from the disease, were discharged from the hospitals. Obviously, prior to recovering, they were diagnosed with the disease and hospitalized due to their severe symptoms.

\paragraph{Cumulative number of deaths}
The data on the cumulative number of diagnosed cases is considered to be inclusive of the diagnosed cases from the French retirement homes (EHPAD). However, the case for data on the cumulative number of deaths is different. Those who died at the hospitals and those who died in the retirement homes (EHPAD) are considered to be distinct. Thus, the cumulative number of deaths is the sum of both data, which are illustrated in Figure~\ref{fig:deaths}.

\begin{figure}[!htb]
    \begin{minipage}{\linewidth}
    \centering
    \includegraphics[width=0.9\textwidth]{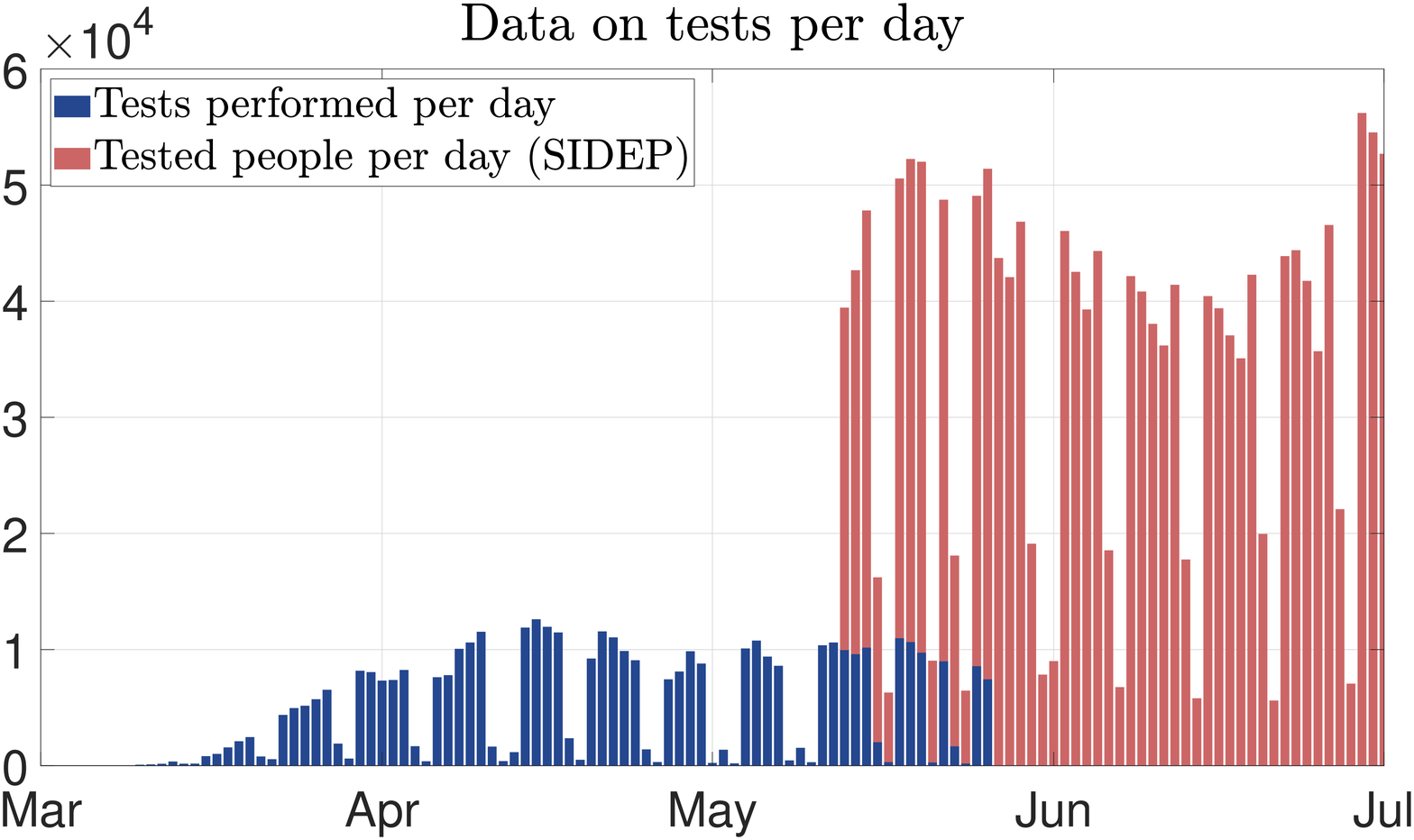}
    \\ (a) 
    \end{minipage}
    
     \begin{minipage}{\linewidth}
    \centering
    \includegraphics[width=0.9\textwidth]{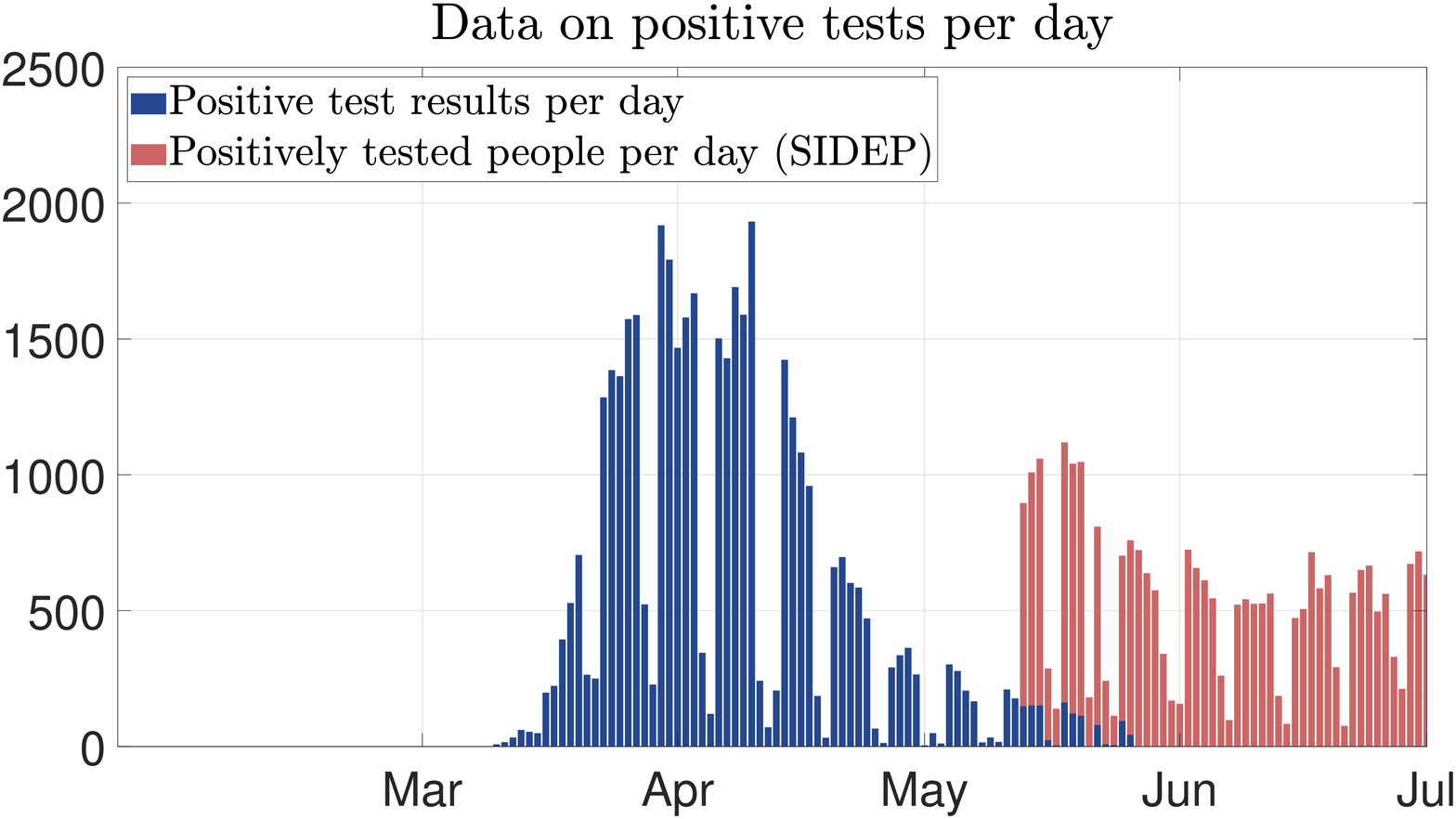}
    \\ (b) 
    \end{minipage}
\caption{Data on the PCR tests: (a) The number of tests performed per day from March 10 to May 26 and the number of tested people per day from May 13 to July 01; (b) The number of positive test results per day from March 10 to May 26 and the number of positively tested people per day from May 13 to July 01. Source: SPF.}
\label{fig:raw_1-tests}
\end{figure}

\paragraph{Number of tests and positive tests per day}

We have two types of data related to COVID-19 PCR tests. 
The first type of data is collected by SPF on the number of tests performed and positive test results per day from March 10 to May 26, 2020. However, this data is collected only from the central sampling laboratories: Eurofins Biomnis and Cerba. Figure~\ref{fig:raw_1-tests}(a) (blue) illustrates the number of tests performed per day and Figure~\ref{fig:raw_1-tests}(b) (blue) illustrates the number of positive test results per day.

The second type of data was made available after the deployment of a new information screening system (SI-DEP) by the SPF. This data is available from May 13, 2020, onward. It is collected from both the laboratories and the hospitals. However, the data reported by SI-DEP is the number of `tested people' per day instead of the number of `tests performed' per day. SI-DEP guarantees that only one test is counted per person. In the case of, for instance, multiple negative test results for a certain person, SI-DEP considers only the first date on which the PCR test was performed. Later, if that person gets a positive test result, then only this new result is reported in the data and the previous data is erased. Figure~\ref{fig:raw_1-tests}(a) (red) illustrates the number of tested people per day and Figure~\ref{fig:raw_1-tests}(b) (red) illustrates the number of positively tested people per day.

\subsection{Imputed data}
 
In the raw data, we only have the data for those who recover or die in the hospitals after being diagnosed with the disease. However, the removed compartment of the SIDUR model also comprises the diagnosed cases who were not hospitalized but were quarantined in their homes. There is no data that records the recovery of these people. Moreover, the data on PCR tests is also incomplete; there is no data on PCR tests from January 24 to March 09, 2020, and the data from March 10 to May 12, 2020, doesn't include the tests performed in the hospitals. Therefore, in order to infer the missing data, we impute the raw data by making reasonable assumptions.

\subsubsection{Cumulative number of removed cases}
From the data on the total number of recovered people from hospitals shown in Figure~\ref{fig:data_hospitals}, we see that 76,540 people have recovered from the hospitals as of July 01, 2020. If we subtract this number and the total number of deaths (Figure~\ref{fig:deaths}), i.e., 29,860, from the total number of diagnosed cases (Figure~\ref{fig:raw_diagnosed}), i.e., 165,700, we obtain $165,700-76,540-29,860=59,300$ people. Further subtracting the currently hospitalized cases (Figure~\ref{fig:data_hospitals}), i.e., 8336 as of July 01, we obtain $59,300-8336=50,964$ people, who might still be infected or have recovered. These people were diagnosed but were not hospitalized; they were quarantined in their homes. However, there is no data that provides a correct answer for how many people have recovered and how many of them are still infected. Therefore, using the relevant raw data, we infer the cumulative number of removed cases $\ol{y}_{2}$ by estimating the number of diagnosed cases who recovered from home.

We use the following notations for simplicity and brevity:

\begin{center}
	\begin{tabular}{ll}
	$\ol{y}_{1}'$ & Total diagnosed and hospitalized \\
	$\ol{y}_{1}''$ & Total diagnosed but not hospitalized \\
	$\ol{y}_{2}'$ & Total recovered/died in a hospital \\
	$\ol{y}_{2}''$ & Total recovered from home after diagnosis
	\end{tabular}
\end{center}
By definition, we have
\be \label{eq:Total_DA}
\ba{ccl}
\ol{y}_{1}(k) &=& \ol{y}_{1}'(k) + \ol{y}_{1}''(k) \\ 
\ol{y}_{2}(k) &=& \ol{y}_{2}'(k)+\ol{y}_{2}''(k)
\ea
\ee
where $k$ is from January 24 to July 01, 2020. Note that $\ol{y}_{1}(k)$ is illustrated in Figure~\ref{fig:raw_diagnosed} and
\[
\ol{y}_{1}'(k) = \ol{y}_{2}'(k) + \ol{H}(k),
\]
where $\ol{y}_{2}'(k)$ is the sum of the total number of recovered cases from hospitals (Figure~\ref{fig:data_hospitals}) and the total number of deaths (Figure~\ref{fig:deaths}), and $\ol{H}$ is the number of active hospitalized cases (Figure~\ref{fig:data_hospitals}). Thus, we can compute the total diagnosed cases who were not hospitalized as
\[
\ol{y}_{1}''(k) = \ol{y}_{1}(k) - \ol{y}_{1}'(k).
\]

Since there is no data for the diagnosed people who recovered from home, therefore $\ol{y}_{2}''(k)$ is unknown. Thus, we assume the following:
\be \label{eq:A_Home}
\frac{\ol{y}_{2}''(k)}{\ol{y}_{1}''(k)} = \frac{\ol{y}_{2}'(k)}{\ol{y}_{1}'(k)}.
\ee
That is, the ratio of the diagnosed cases who recovered in homes to the total diagnosed cases who were quarantined at homes is equal to the ratio of the diagnosed cases who recovered or died in hospitals to the total diagnosed cases who were hospitalized. In other words, we assume that the removal rate of people who were not hospitalized is equal to the removal rate of people who were hospitalized. Thus, from \eqref{eq:Total_DA} and \eqref{eq:A_Home}, we obtain
\[
\ol{y}_{2}(k) = \ol{y}_{2}'(k) \left( 1 + \frac{\ol{y}_{1}''(k)}{\ol{y}_{1}'(k)}\right) = \ol{y}_2'(k) \frac{\ol{y}_1(k)}{\ol{y}_1'(k)}.
\]
which corresponds to the model output $y_2(t)$ in \eqref{total_recovered_cases}.

\subsubsection{Combining two types of testing data}

From Figure~\ref{fig:raw_1-tests}, we see that the first type of data, which is available from March 10 to May 26, 2020, considers the number of tests performed and positive test results per day. On the other hand, the second type of data, which is available from May 13, 2020, onward, considers the number of tested people and positively tested people per day. However, no person is usually tested more than once per day. Therefore, we assume that the number of tested people per day is same as the number of tests performed per day. Similarly, the number of positively tested people per day is same as the number of positive test results per day. 
Note that if a person is tested more than once but on different days, then this assumption is not violated.


We consider three time intervals: 
(i)~January 24--March 09, when there is no data on PCR tests; (ii)~March 10--May 12, when there is incomplete data; (iii)~May 13--July 01, when there is complete data.
Let $\ol{u}$ and $\ol{y}_3$ denote the number of tests performed and the number of positive test results per day, respectively, for the entire interval January 24 to July 01, 2020. Let $\ol{u}',\ol{u}'',\ol{u}'''$ and $\ol{y}_3',\ol{y}_3'',\ol{y}_3'''$ be the number of tests performed and the number of positive test results obtained for the first, second, and third time intervals, respectively.
Since the data in the third interval is reliable, we do not make any imputations for $\ol{u}'''$ and $\ol{y}_3'''$. For the other two intervals, we make reasonable assumptions to complete the data.

\begin{enumerate}[(i)]
\item {\it January 24--March 09:} This interval corresponds to the beginning of the epidemic in France and the data for tests performed and positive test results per day for this interval is $\ol{u}'$ and $\ol{y}_3'$, respectively. During this interval, only those people were tested who showed symptoms. Moreover, recall the output relation \eqref{output_relation}. Then, we compute the data as follows:
$
\ol{u}'(k) \approx \ol{y}_3'(k)
$ and
$
\ol{y}_3'(k) = \ol{y}_1(k+1) - \ol{y}_1(k)
$, where $k$ is from January 24 to March 09.
In other words, during the first interval, the number of tests performed per day is assumed to be approximately equal to the number of positive test results obtained per day. Moreover, the number of positive test results obtained per day is equal to the number of diagnosed cases that day.

\item {\it March 10--May 12:} In the second interval, we have the data on PCR tests that is reported only by the laboratories and not by the hospitals. During this interval, we compute the data as follows:
$\ol{u}''$ is same as the data (Figure~\ref{fig:raw_1-tests}) and $\ol{y}_3''(k) = \ol{y}_1(k+1)-\ol{y}_1(k)$, where $k$ is from March 10 to May 12. 
\end{enumerate}
Based on the above data imputations, we obtain the number of tests performed per day $\ol{u}=\left[\ba{ccc} \ol{u}' & \ol{u}'' & \ol{u}''' \ea\right]$, which corresponds to the control input $u(t)$, and the number of positive tests obtained per day $\ol{y}_3=\left[\ba{ccc} \ol{y}_3' & \ol{y}_3'' & \ol{y}_3''' \ea\right]$, which corresponds to the model output $y_3(t)$ in \eqref{testedpos_cases}, for the complete time interval January 24 to July 01, 2020. 



\begin{figure}[!htb]
	\centering
	\includegraphics[width=0.7\textwidth]{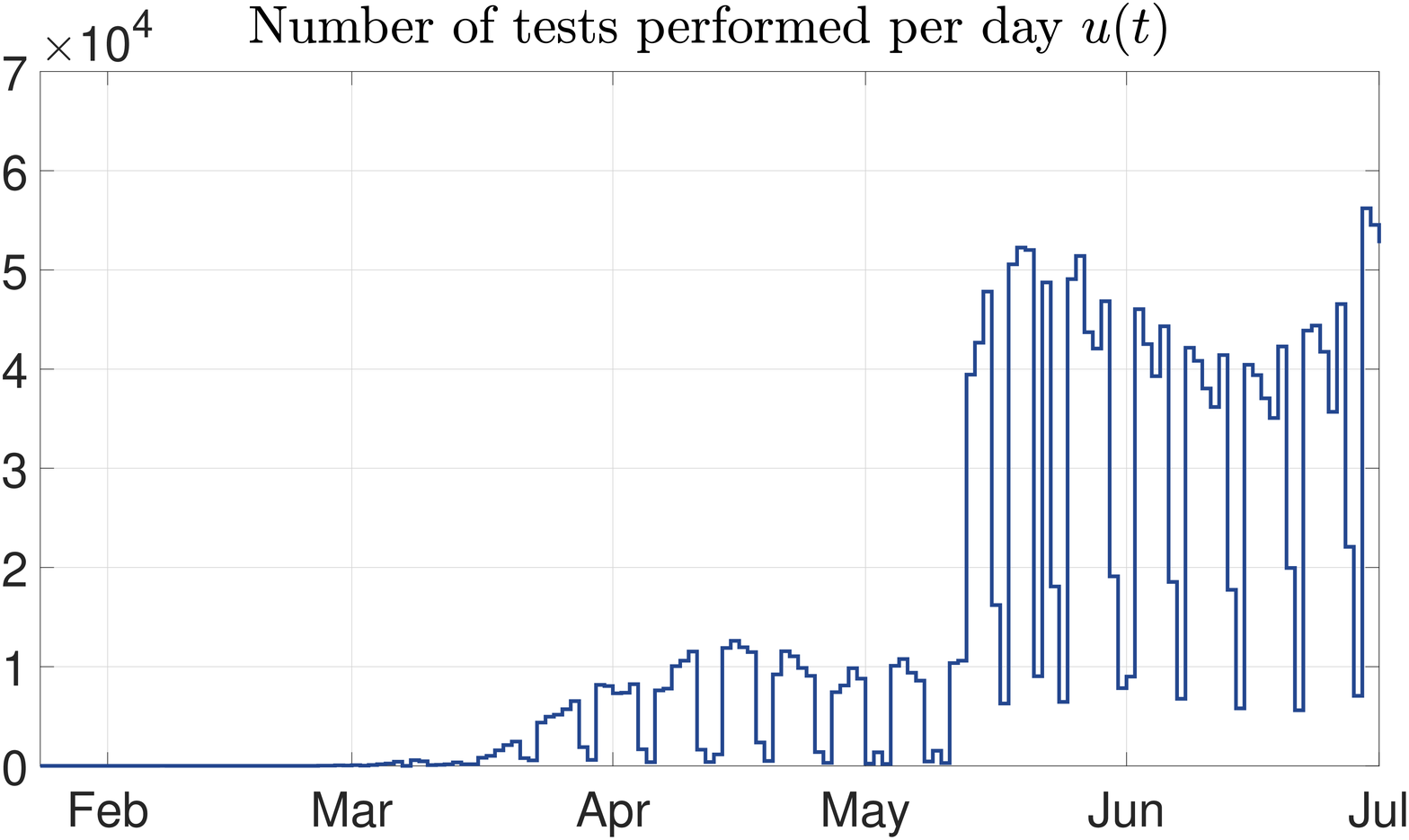}
	\caption{Input signal from the data.}
	\label{fig:input_plot}
\end{figure}

\begin{figure}[!htb]
	\begin{minipage}{0.5\linewidth}
    \centering
    \includegraphics[width=1.0\textwidth]{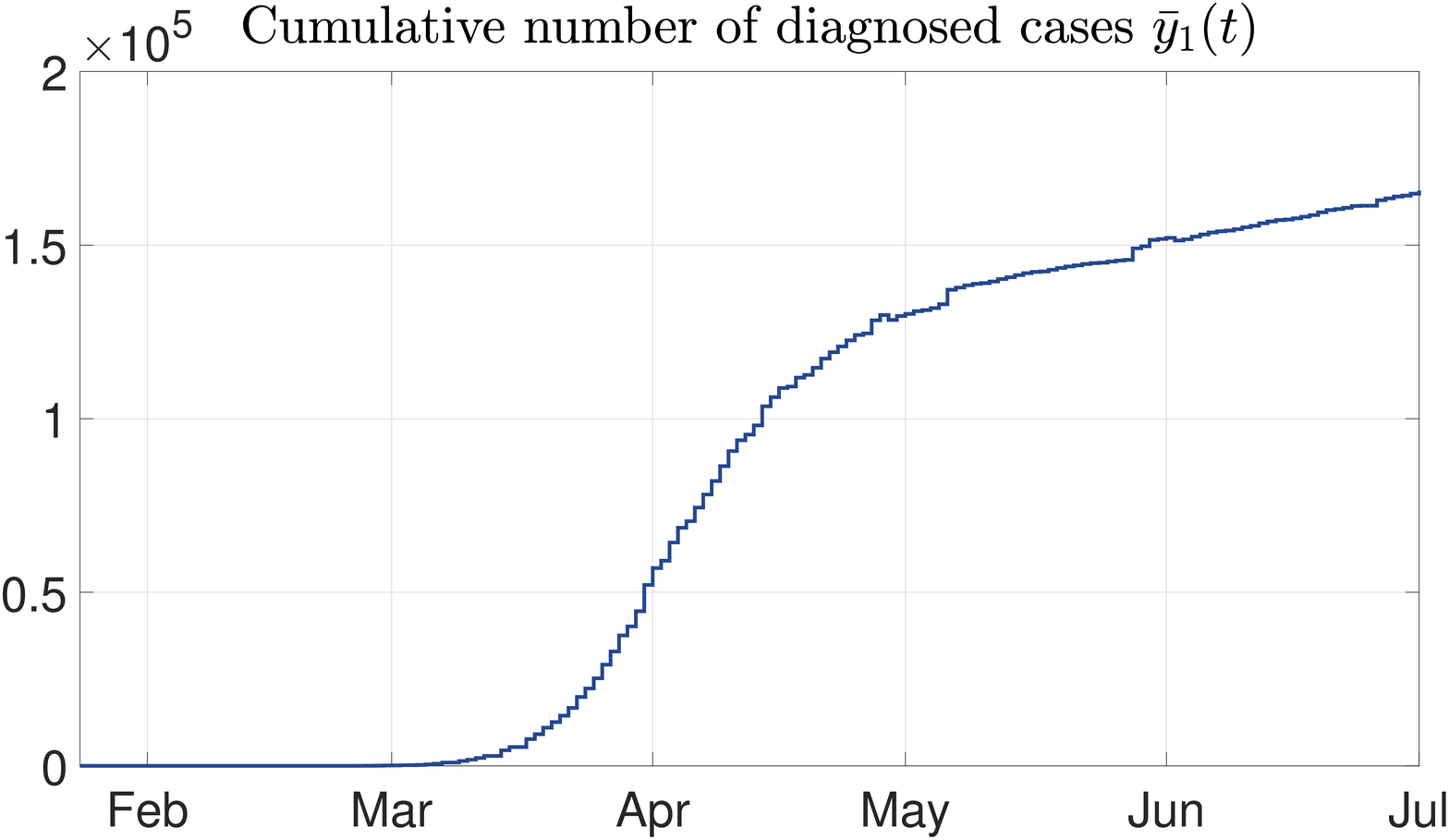}
    \end{minipage}
    \hfill
    \begin{minipage}{0.5\linewidth}
    \centering
    \includegraphics[width=1.0\textwidth]{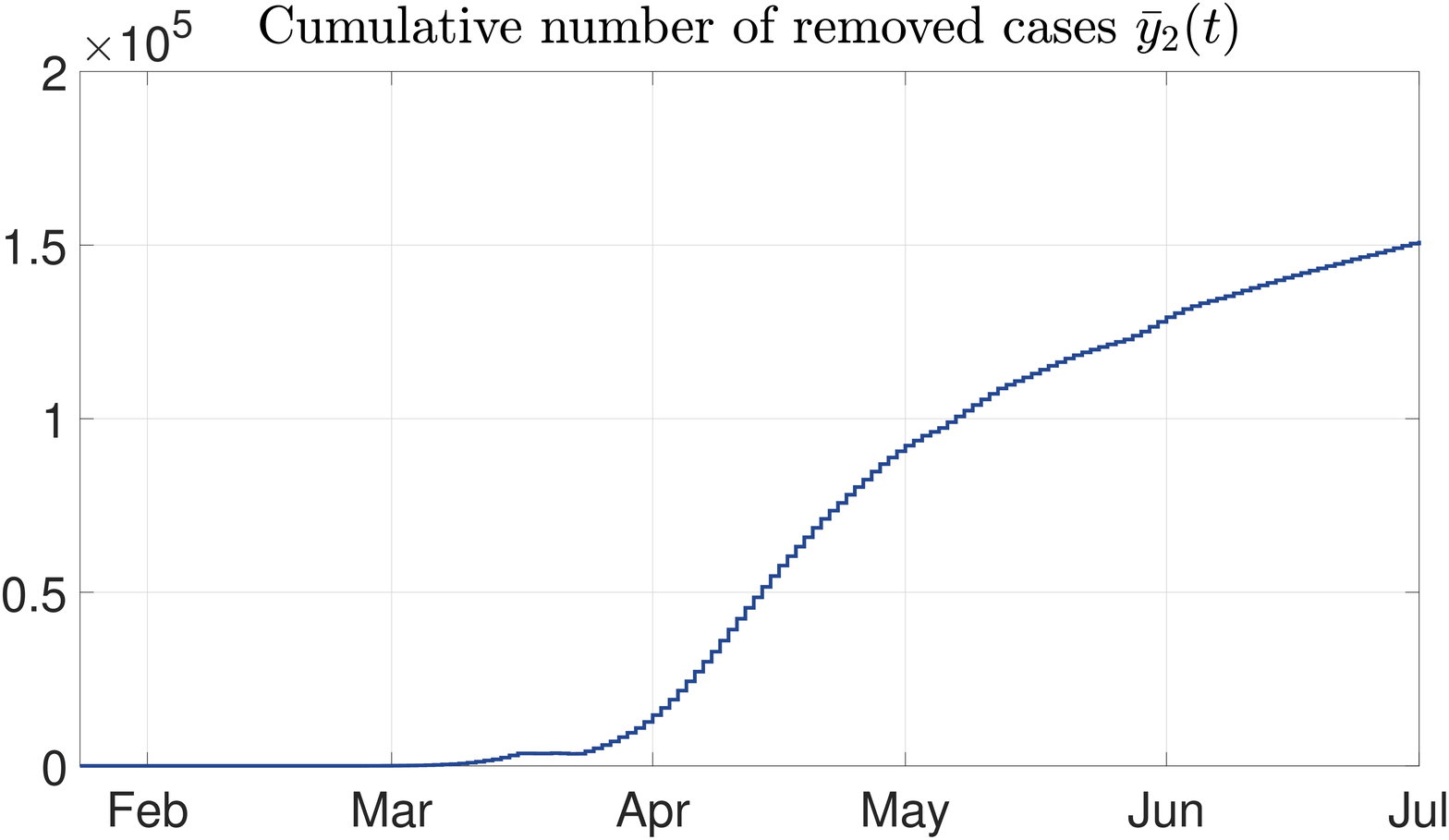}
    \end{minipage}
    
    \vspace*{0.5cm}
    \centering
    \begin{minipage}{0.5\linewidth}
    \centering
    \includegraphics[width=1.0\textwidth]{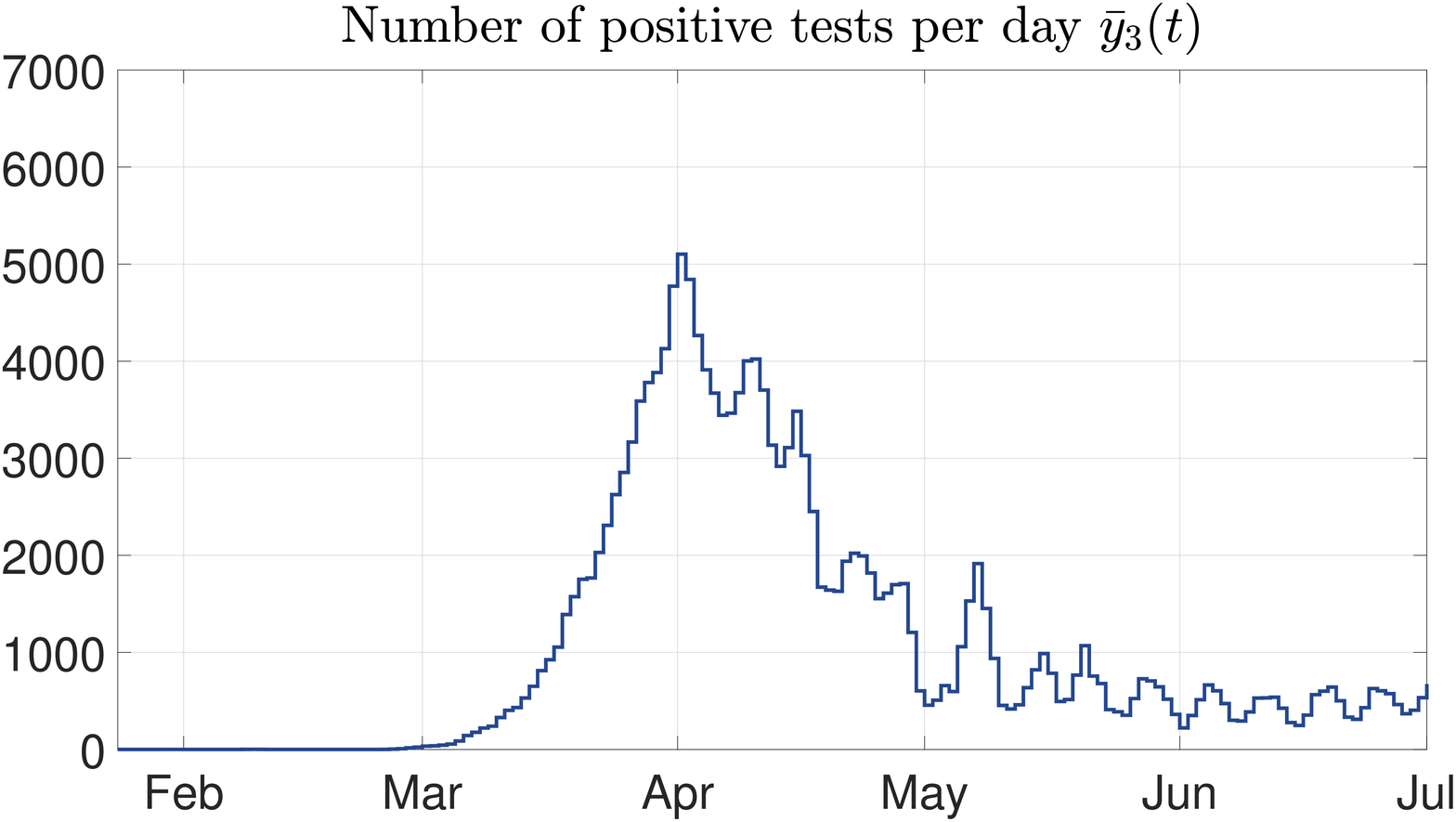}
    \end{minipage}
\caption{Output signals from the data.}
\label{fig:outputs_plot}
\end{figure}

\subsection{Input and output signals from the data}

The input $u(t)$ to the model corresponds to the number of tests performed per day $\ol{u}$. Let $\{1,2,\dots,\tau\}$ be the index set of the 160 days from January 24 to July 01, 2020, where $\tau=160$.
Then, for $k=1,2,\dots,\tau$, we define the input signal as
\[
u(t) = \ol{u}(k), \quad \text{for}~\lfloor t \rfloor \leq k < \lceil t \rceil
\]
which is illustrated in Figure~\ref{fig:input_plot}.

We denote by $\ol{y}_i(t)$, for $i=1,2,3$, the outputs obtained from the data. That is, the output signals from the data and the model are related as follows
\[
\ba{ccl}
\ol{y}_1(t) &=& y_1(t) + w_1(t) \\
\ol{y}_2(t) &=& y_2(t) + w_2(t) \\
\ol{y}_3(t) &=& y_3(t) + w_3(t)
\ea
\]
where $w_i(k)$, for $i=1,2,3$, represents the measurement noise.
The outputs $\ol{y}_1(t)$, $\ol{y}_2(t)$, and $\ol{y}_3(t)$ correspond to the cumulative number of diagnosed cases, the cumulative number of removed cases, and the number of positive test results per day, respectively. Similar to the input $u(t)$, we define the output signals from the data as
\[
\left\{
\ba{ccll}
\ol{y}_1(t) &=& \ol{y}_{1}(k), & \quad \text{for}~\lfloor t \rfloor \leq k < \lceil t \rceil  \\
\ol{y}_2(t) &=& \ol{y}_{2}(k), & \quad \text{for}~\lfloor t \rfloor \leq k < \lceil t \rceil \\
\ol{y}_3(t) &=& \ol{y}_3(k), & \quad \text{for}~\lfloor t \rfloor \leq k < \lceil t \rceil
\ea
\right.
\]
which are illustrated in Figure~\ref{fig:outputs_plot}.

\section{Estimation of model parameters} \label{sec:estimation}
In this section, we validate the SIDUR model by estimating the model parameters $\rho$, $\beta$, $\theta$, and $\gamma$ for the case of COVID-19 in France. 

\subsection{Estimation of $\rho$} \label{section_rho}

The removal rate $\rho$ can be directly estimated from the data outputs $\ol{y}_1$ and $\ol{y}_2$. Consider a daily sampling of the model equation \eqref{eq401e}, which leads to 
\[
\Delta x_\R(k) \approx \rho x_\D(k)
\]
where $\Delta$ stands for the forward difference operator, i.e., $\Delta x_\R(k) = x_\R(k+1)-x_\R(k)$ for some nonnegative integer $k$. Therefore, from the relation between $y_1$ and $y_2$ in \eqref{y2y1_relation}, we obtain
\be \label{eq_LSgamma2}
\Delta \ol{y}_2(k) = \rho\, \ol{y}_{12}(k) + e(k)
\ee
where $\ol{y}_{12}(k) = \ol{y}_1(k) - \ol{y}_2(k)$ and $e(k)$ is the error term due to measurement noise. Then, the problem of estimating $\rho$ can be formulated as follows: Find $\rho^*$ such that
\[
\rho^* = \arg\min_{\rho\in[0,1]} \sum_{k=1}^\tau \|\Delta \ol{y}_2(k) - \rho \, \ol{y}_{12}(k)\|^2.
\]
Notice that the solution of this problem can be obtained through least-square estimation \cite[Chapter 7]{ljung1999}.

\subsection{Estimation of $\beta,\theta,\gamma$} \label{section_beta-theta-gamma}
We formulate a problem of fitting the model outputs $y_1(t),y_2(t),y_3(t)$ to the data outputs $\ol{y}_1(k),\ol{y}_2(k),\ol{y}_3(k)$, where $k=1,2,\dots,\tau$ with $\tau$ being the final time. The model fitting is done by optimizing the parameters $\beta,\theta,\gamma$ for the time interval $[0,\tau]$ under the assumption that $\gamma$ is constant whereas $\beta$ and $\theta$ are piecewise constants.

To limit the rate of spread of COVID-19, the French government announced to place a lockdown all over France from March 17 to May 10, 2020, which included restricted human mobility, strict social distancing measures, and closure of schools, offices, and marketplaces. However, the essential services and public establishments were authorized to remain open under strict preventive measures. People were allowed to leave their homes with face masks only for necessary groceries, brief exercise within a certain radius of their homes, or for urgent medical reasons. Such an intervention from the public authority is necessary to mitigate the rate of spread of the disease and to reduce the value of infection rate $\beta$. Therefore, in relation to the case of France, we divide the time into three intervals: (i)~Before lockdown (January 24 to March 16), (ii)~During lockdown (March 17 to May 10), and (iii)~After lockdown (May 11 to July 01). We consider a different value of the infection rate $\beta$ during each of these intervals, i.e.,
\[
\beta(k) = \left\{\ba{ll}
\beta_1, & \text{for}~k=\text{January 24 to March 16} \\
\beta_2, & \text{for}~k=\text{March 17 to May 10} \\
\beta_3, & \text{for}~k=\text{May 11 to July 01}
\ea\right.
\]
where $\beta_1,\beta_2,\beta_3$ are positive real numbers.
For the testing specificity parameter $\theta$, we divide the time into two intervals: (i)~Before May 11 and (ii) After May 11, where May 11 corresponds to the change in testing policy in France \cite{hale2020b} (also see the website of Our-World-in-Data\footnote{Website: \href{https://ourworldindata.org/grapher/covid-19-testing-policy?time=2020-05-11&region=Europe}{Our World in Data: COVID-19 Testing Policies.} (Accessed 30/09/2020)}). Thus, we have
\[
\theta(k) = \left\{\ba{ll}
\theta_1, & \text{for}~k=\text{January 24 to May 10} \\
\theta_2, & \text{for}~k=\text{May 11 to July 01}
\ea\right.
\]
where $\theta_1$ and $\theta_2$ are real numbers in the interval $[0,1]$.

Let $p = [\ba{cccccc} \beta_1 & \beta_2 & \beta_3 & \theta_1 & \theta_2 & \gamma \ea]^T$ be the parameter vector. Then, the goal is to find $p^*$ such that
\be \label{prob:param_est}
p^* = \arg\min_{p} \mc{J}(p)
\ee 
where the cost function is given by
\be \label{eq:cost_fun}
\mc{J}(p) = \sum_{k=1}^{\tau} \left[ \left(y_1(k,p)-\ol{y}_1(k)\right)^2 + \left(y_2(k,p)-\ol{y}_2(k)\right)^2 + \left(y_3(k,p)-\ol{y}_3(k)\right)^2 \right]
\ee
with the model outputs $y_i$, $i=1,2,3$, depending on the parameter vector $p$.
Note that we consider the data from January 24 to July 01, 2020, therefore we have $\tau = 160$ days.

To solve this problem, one can also pose it as a least-square estimation, as we did for the removal rate $\rho$, by defining relations between the data outputs $\ol{y}_1,\ol{y}_2,\ol{y}_3$. However, such relations include the difference operator $\Delta$ applied twice to the data outputs, which is usually not recommended when the data is noisy because it amplifies the measurement noise. Moreover, the gradient-based estimation algorithms \cite[Chapter 4]{nelles2001} are also not suitable due to the difficulty of computing the gradient of the cost function $\mc{J}$ online with respect to the parameter vector $p$. This is because the model outputs $y_1,y_2,y_3$ do not depend directly on the parameters but through the solution trajectories of the SIDUR model. For simplicity, therefore, we choose the particle swarm optimization (PSO) \cite{kennedy1995} described in \ref{appendix_PSO}, which is a `derivative-free' algorithm, to estimate the parameter vector $p$.


\begin{table}[!htb]
    \centering
    \begin{tabular}{|l||l|}
        \hline
        Infection rate & $\beta_1 = 0.3708 \quad \beta_2 = 0.0707 \quad \beta_3 = 0.3717$  \\
        \hline
        Testing specificity & $\theta_1 = 0.9948 \quad \theta_2 = 0.9967$  \\
        \hline
        Recovery rate & $\gamma = 0.1589$   \\
        \hline
        Removal rate & $\rho = 0.0499$ \\ 
        \hline
    \end{tabular}
    \caption{Estimated parameter values.}
    \label{tab:estimated_parameters}
\end{table}

The estimated parameter values are provided in Table~\ref{tab:estimated_parameters}. The estimated recovery rate $\gamma$ and removal rate $\rho$ show that an undiagnosed person recovers in an average period of about $6.3$ days and a diagnosed person recovers or dies in an average period of about $20$ days. The testing specificity parameter changes slightly from $\theta_1=0.9948$ to $\theta_2=0.9967$, which can have significant impact on the positive test results because it multiplies with the sum of the susceptible and unidentified recovered population in \eqref{eq:testable} that is in the order of $10^7$ in the case of France.

The infection rate $\beta$ changes its value twice. First, it drops from $\beta_1=0.3708$ to $\beta_2=0.0707$ when the lockdown is implemented in France on March 17, which significantly decreased the rate of the epidemic spread. Then, it rises from $\beta_2=0.0707$ to $\beta_3=0.3717$ when the lockdown is lifted on May 10. Many restrictions like social distancing and wearing of face masks were still in place after May 10 in order to prevent the spread of COVID-19 in France. However, the increase in the value of $\beta$ can be explained by the summer vacations when people were allowed to travel everywhere across  France and Europe\footnote{Website: \href{https://www.sortiraparis.com/actualites/a-paris/articles/215179-coronavirus-vacances-d-ete-partout-en-france-et-en-europe}{Sortir \`{a} Paris: Summer holidays in France and Europe.} (Accessed 30/09/2020)}. This made the places with tourist attractions very crowded and resulted in a higher infection rate.

\subsection{Model fitting}


Using the estimated values of the model parameters in Table~\ref{tab:estimated_parameters}, we run the model from January 24 to July 01, 2020. The model fits the output signals data as shown in Figure~\ref{fig:validation}.

\begin{figure}[!htb]
\centering
	\begin{minipage}{0.9\textwidth}
	\centering
	\includegraphics[width=\textwidth]{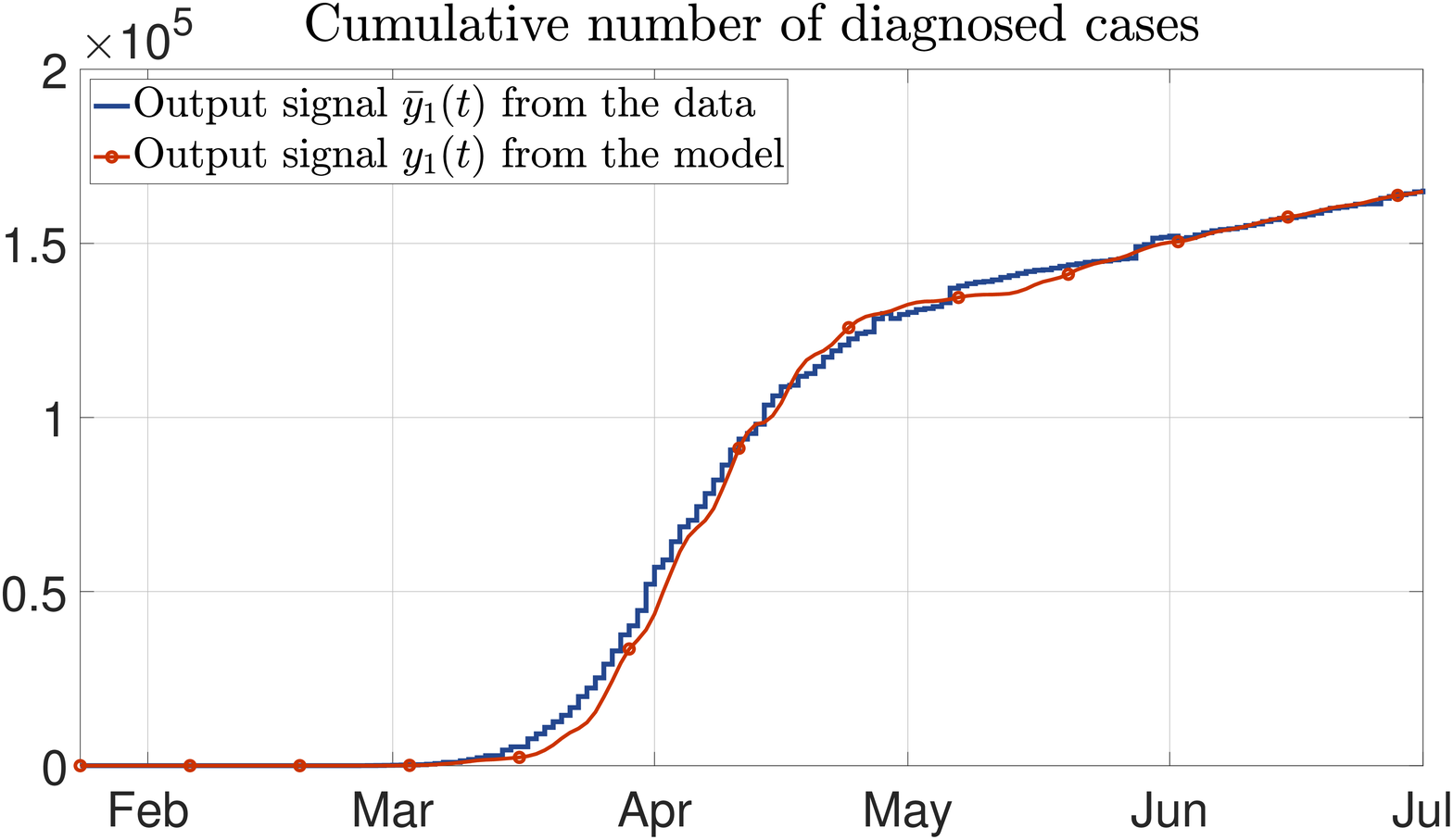}
	\end{minipage}
	
	\vspace*{0.5cm}
	\begin{minipage}{0.9\textwidth}
	\centering
	\includegraphics[width=\textwidth]{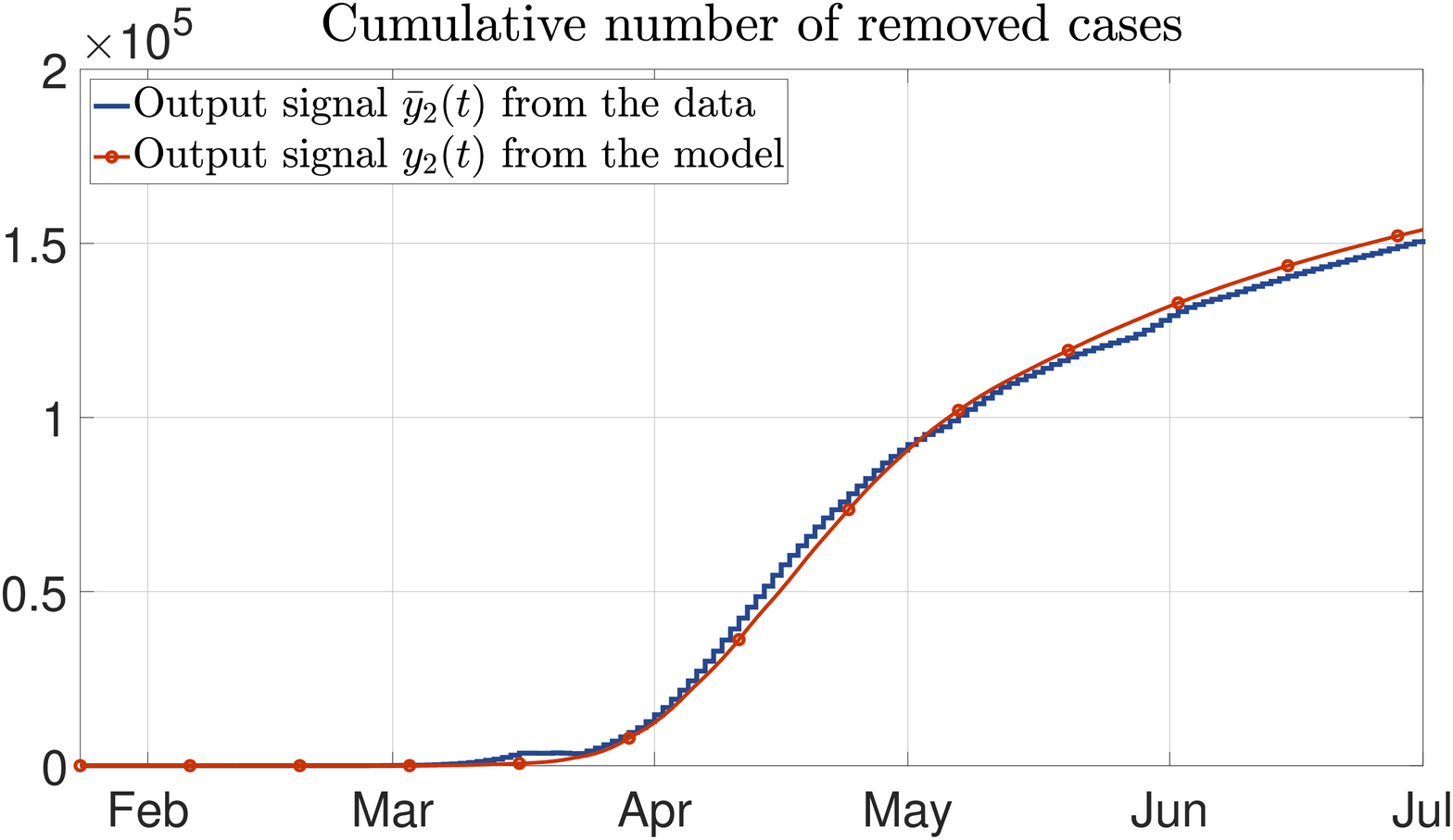}
	\end{minipage}
	
	\vspace*{0.5cm}
	\begin{minipage}{0.9\textwidth}
	\centering
	\includegraphics[width=\textwidth]{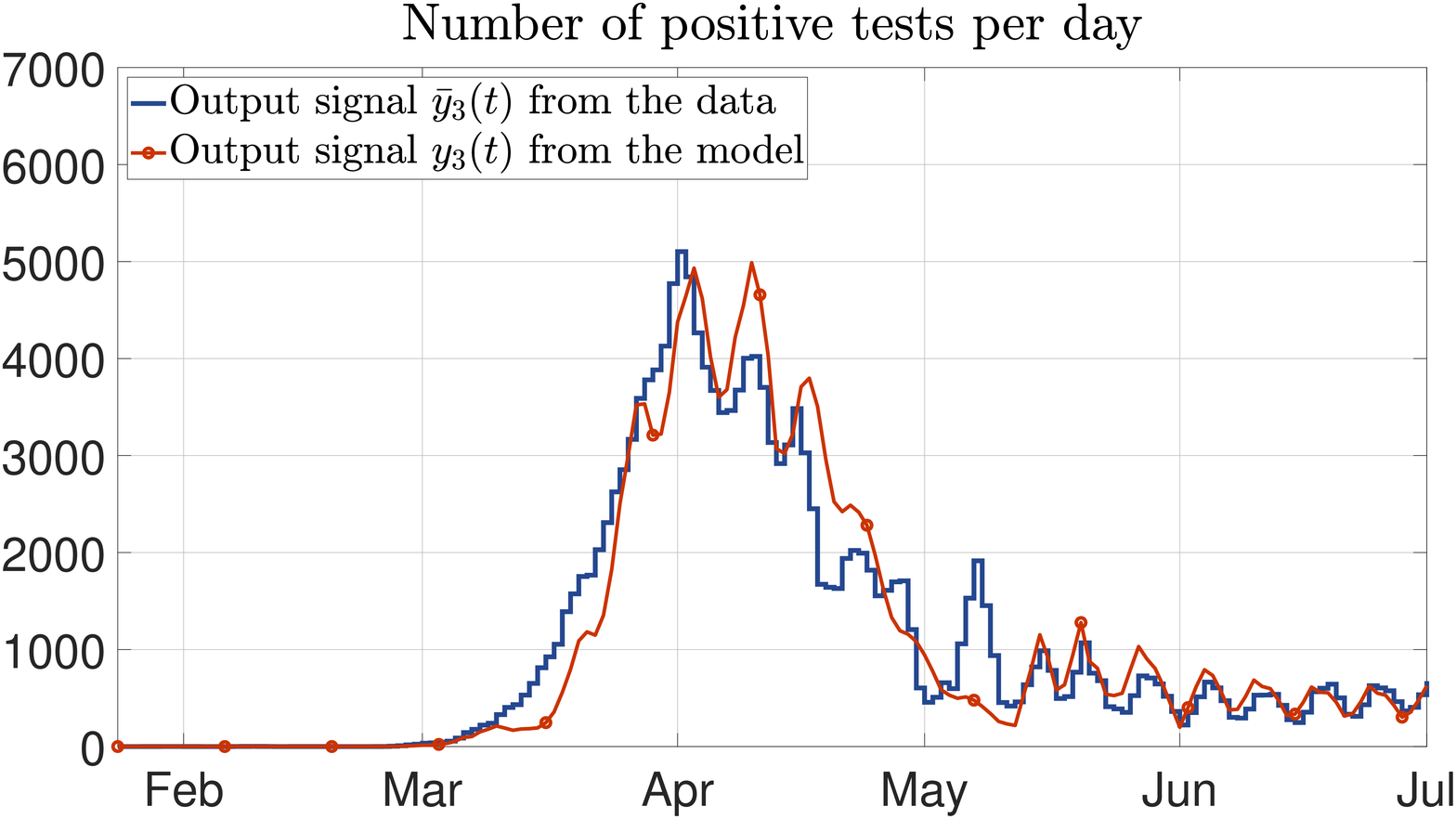}
	\end{minipage}
	
	\caption{Model fitting.}
	\label{fig:validation}
\end{figure}


The basic reproduction number $R_0$ at the outbreak of the COVID-19 epidemic in France is computed using \eqref{eq:basic_R0} and the average value of effective reproduction number $R_t$ during three phases (before, during, and after lockdown) are computed using \eqref{eq:eff_R}. For the `after lockdown' phase, we chose July 01, 2020, to compute $R_t$ because it is the date up to which our data is considered. These computed values are shown in Table~\ref{tab:reproduction_numbers} with a comparison to the ones reported by the government\footnote{Website: \href{https://www.santepubliquefrance.fr/maladies-et-traumatismes/maladies-et-infections-respiratoires/infection-a-coronavirus/documents/bulletin-national/covid-19-point-epidemiologique-du-17-septembre-2020}{COVID-19 France: Epidemiological update.} (Accessed 30/09/2020)} \footnote{Website: \href{https://dashboard.covid19.data.gouv.fr/suivi-indicateurs?location=FRA}{COVID-19 France: Monitoring of the epidemic.} (Accessed 30/09/2020)}.

\begin{table}[!htb]
    \centering
    \begin{tabular}{|l||l|l|}
        \hline
        Epidemic phases &  Computed from model & Reported by French government \\
        \hline
        Outbreak & $R_0=2.33$ & $R_0=2.7$ \\
        \hline
        Before lockdown & $R_t = 2.3$ & $R_t=2.7$\\
        \hline
        During lockdown & $R_t = 0.33$ & $R_t=0.7$ \\ 
        \hline
        After lockdown & $R_t = 1$ & $R_t=1$ \\
        \hline
    \end{tabular}
    \caption{The basic reproduction number $R_0$ at the outbreak and the values of the effective reproduction number $R_t$ at the end of each phase of the COVID-19 epidemic in France. The values computed by our model are quite close to the ones reported by the French government.}
    \label{tab:reproduction_numbers}
\end{table}

The change in the value of $R_t$ also influences the evolution of diagnosed population $x_\D(t)$. This is because larger value of $R_t$ results in a larger infected population $x_\I(t)$ and smaller value of $R_t$ results in a smaller infected population $x_\I(t)$, which respectively increases and decreases the probability of detecting an infected person $x_\I(t)/x_\T(t)$ by a single test. Keeping the number of tests performed per day same, the larger probability of detection $x_\I(t)/x_\T(t)$ results in a larger diagnosed population $x_\D(t)$.

In Table~\ref{tab:reproduction_numbers}, we see that the placement and lifting of lockdown on March~17 and May~11, respectively, had a significant impact on the value of $R_t$. Such an effect on $R_t$ impacted the evolution of the diagnosed population $x_\D(t) = y_1(t) - y_2(t)$, which can be interpreted as the number of active confirmed cases and is illustrated in Figure~\ref{fig:validation2}. The placement of lockdown reduced the value of $R_t$ and resulted in less number of active confirmed cases as compared to the scenario in Figure~\ref{fig:validation2} where the lockdown was not placed on March~17. In this scenario, as shown in Figure~\ref{fig:validation2}, the number of active confirmed cases would have increased to a point that could have challenged the available medical facilities such as hospital beds, ventilators, and ICUs. On the other hand, the lifting of lockdown increased the value of $R_t$ and resulted in more number of active confirmed cases as compared to the scenario where the lockdown was not lifted on May~11.

\begin{figure}[!htb]
    \centering
    \includegraphics[width=0.9\textwidth]{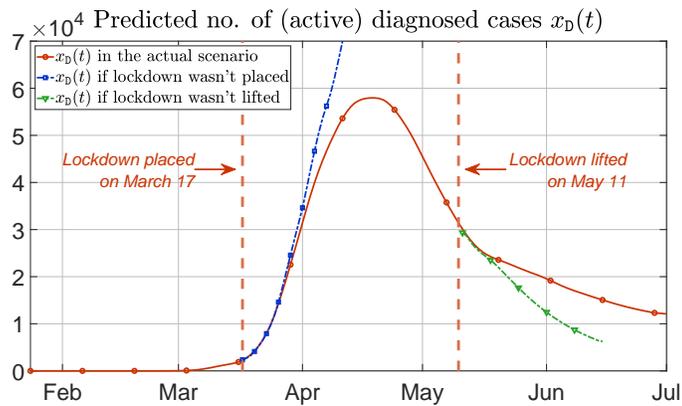}
    \caption{The comparison between the number of active diagnosed cases in the actual scenario vs. two scenarios: if the lockdown was not placed on March 17 and if the lockdown was not lifted on May 11.}
    \label{fig:validation2}
\end{figure}

\subsection{Number of active ICU patients and deaths}


The number of active ICU patients $B(t)$ is a function of the number of active infected people $A(t)$. Since an infected person starts to show symptoms after the average incubation period of approximately $5$ days, \cite{lauer2020}, and a person takes on average 12 days from being diagnosed to being addmitted to ICU, \cite{zhou2020}, we assume $\psi^{-1} = 5 + 12=17$ days to be the average time delay from getting infected to being admitted to ICU for a typical COVID-19, critically ill case. Thus, we model the number of active ICU patients $B(t)$ as a function of $A(t-\psi^{-1})$, which is approximated by:
\be \label{eq:icu-infection}
B(t) =  b_1 A(t-\psi^{-1}) + b_2 \sqrt{A(t-\psi^{-1})}
\ee
where $b_1$ and $b_2$ are the parameters (\ref{appendix_icu-death}, Table~\ref{tab:param_death-ICU}) that are determined via the least-square solution to fit \eqref{eq:icu-infection} to the data on the number of ICU patients. This is illustrated in Figure~\ref{fig:ICU-Infection}.

\begin{figure}[!htb]
    \centering
    \includegraphics[width=0.9\textwidth]{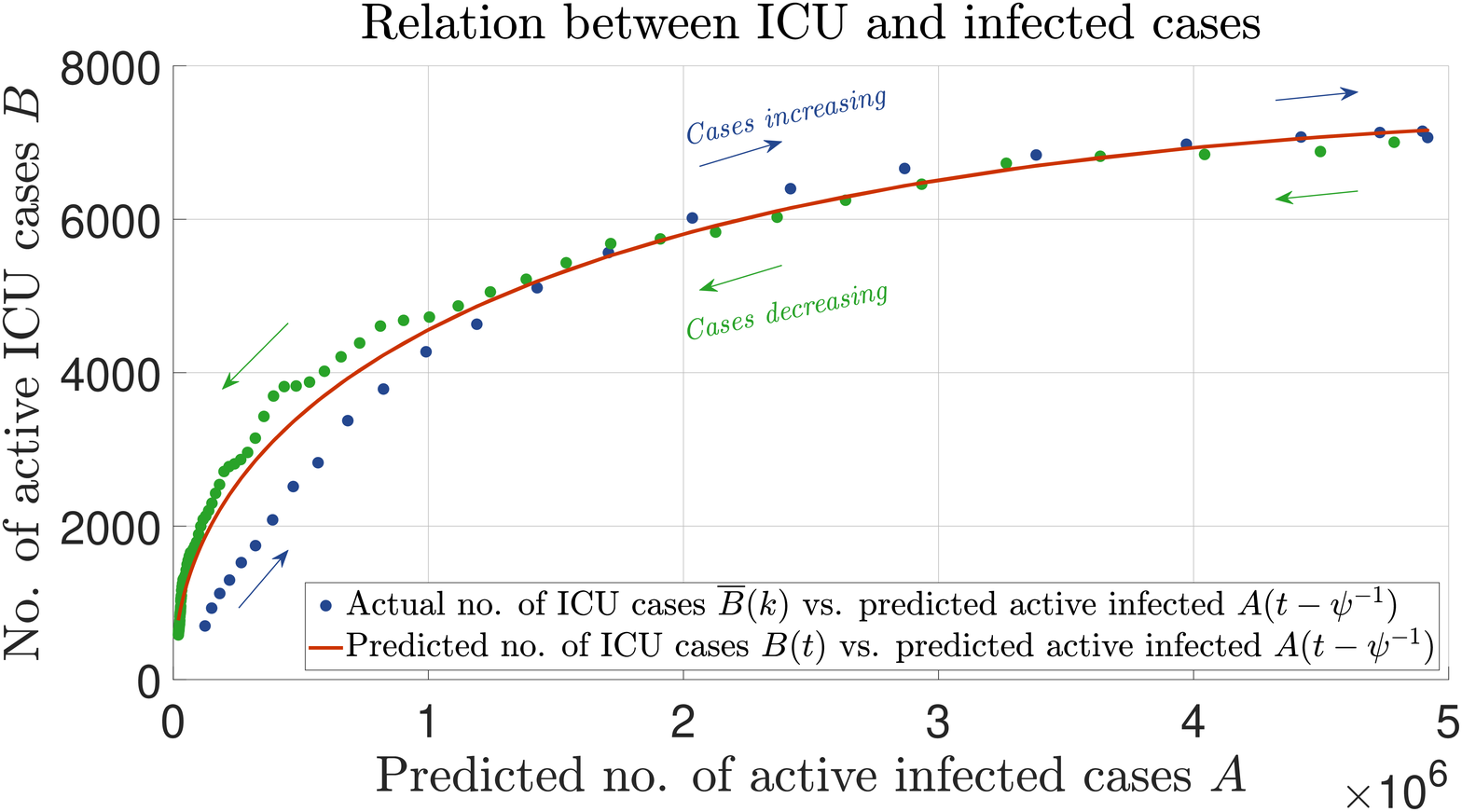}
    \caption{Number of active ICU patients $B(t)$ with respect to the number of active infected cases $A(t)$.} 
    \label{fig:ICU-Infection}
\end{figure}

\begin{figure}[!htb]
    \centering
    \includegraphics[width=0.9\textwidth]{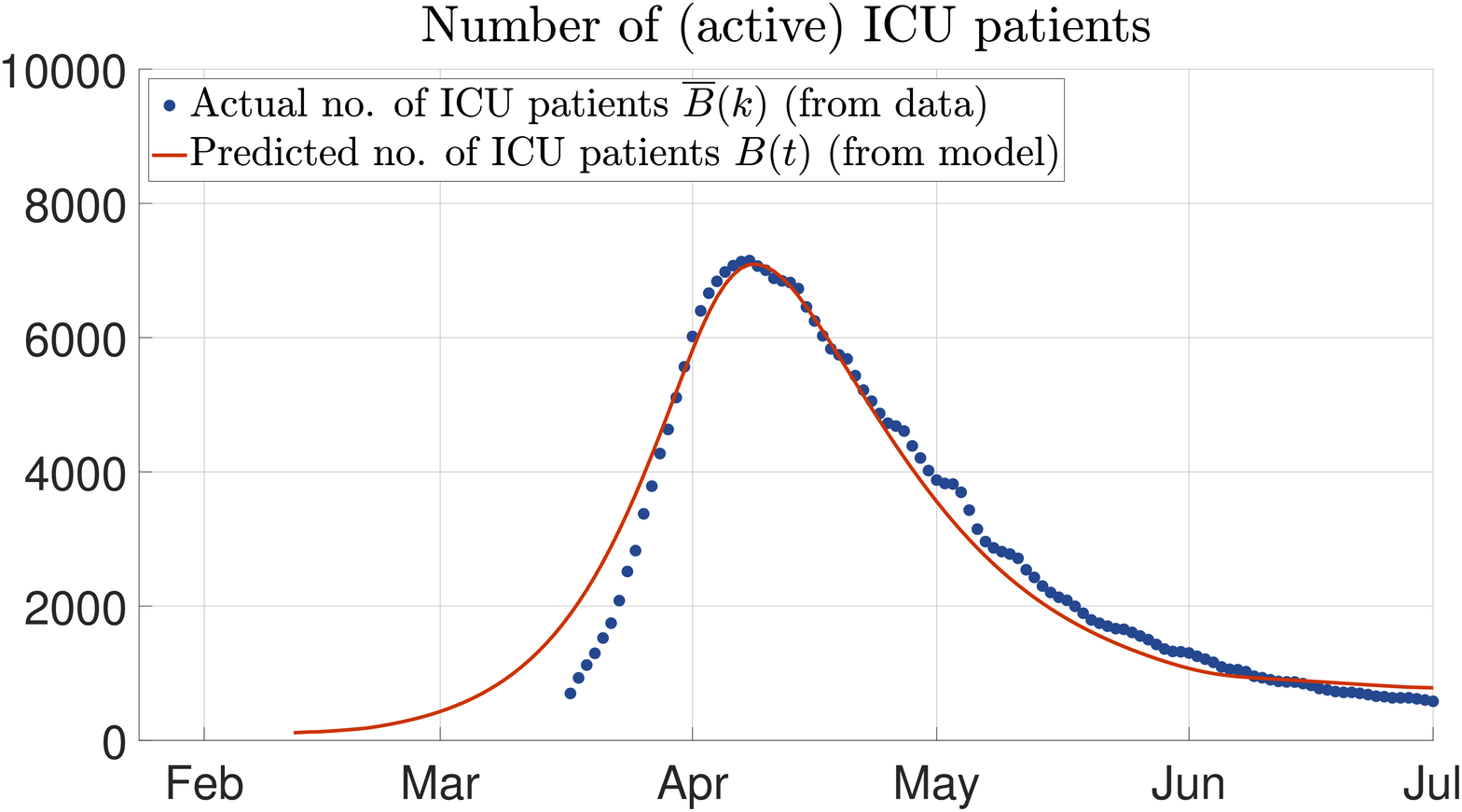}
    \caption{Model fit of the data on the number of active COVID-19 ICU cases $B(t)$ in France using the relation \eqref{eq:icu-infection}.}
    \label{fig:ICU-fit}
\end{figure}
\begin{figure}[!htb]
    \centering
    \includegraphics[width=0.9\textwidth]{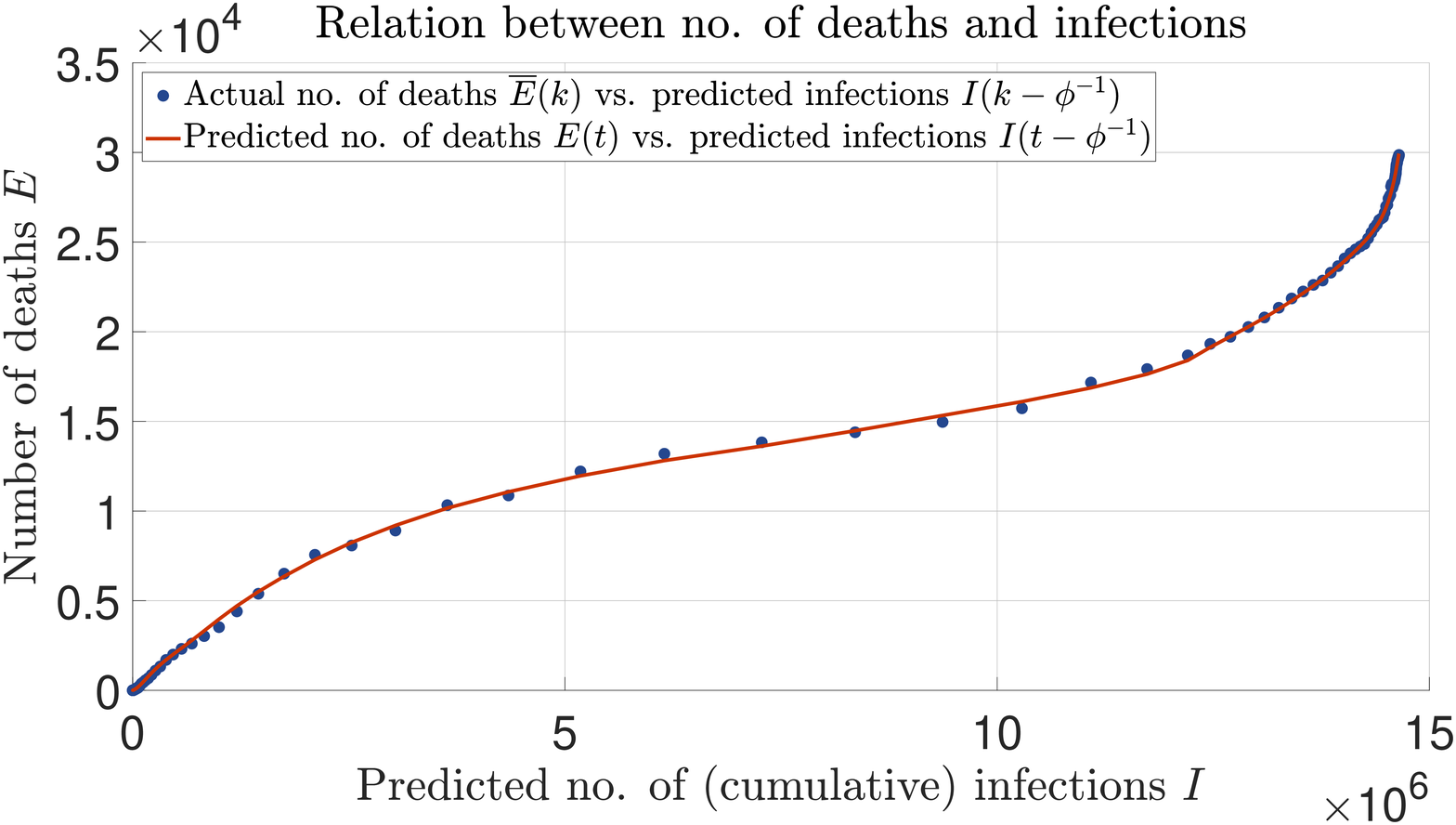}
    \caption{Cumulative number of deaths $E(t)$ with respect to the cumulative number of infected cases $I(t)$.}
    \label{fig:Death-Infection}
\end{figure}
\begin{figure}[!htb]
    \centering
    \includegraphics[width=0.9\textwidth]{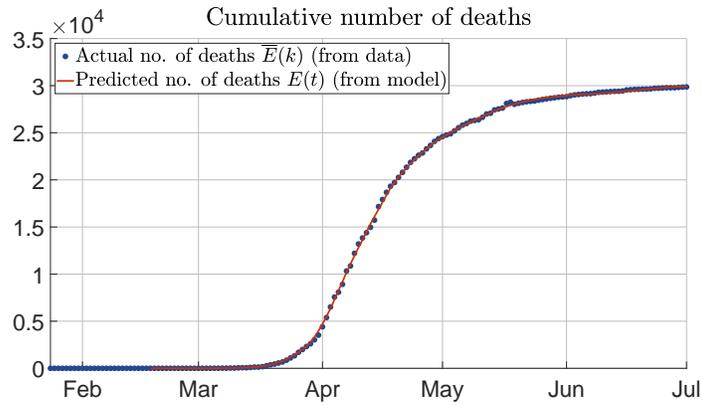}
    \caption{Model fit of the data on the number of COVID-19 deaths $E(t)$ in France using the relation \eqref{eq:death-infection}.}
    \label{fig:Death-fit}
\end{figure}

Similar to the case of the number of active ICU patients, a typical non-surviving case has an average incubation period of $5$ days and, in addition to that, an average removal period of $\rho^{-1}\approx 20$ days, where $\rho$ is the removal rate, Table~\ref{tab:estimated_parameters}. Thus, assuming $\phi^{-1}=5+\rho^{-1}=25$ days to be the average time delay from getting infected to death of a typical non-surviving COVID-19 case, we model the number of deaths $E(t)$ as a function of $I(t-\phi^{-1})$, which is approximated by the following polynomial:
\be \label{eq:death-infection}
E(t) = \sum_{i=1}^{10} e_i I^i(t - \phi^{-1})
\ee
where $e_i$, for $i=1,\dots,10$, are the parameters (\ref{appendix_icu-death}, Table~\ref{tab:param_death-ICU}) that are determined via the least-square solution to fit \eqref{eq:death-infection} to the data on the number of deaths. This is illustrated in Figure~\ref{fig:Death-Infection}. 

Using the relations \eqref{eq:icu-infection} and \eqref{eq:death-infection}, we illustrate the model fit of the number of active ICU cases and the cumulative number of deaths with the data in Figure~\ref{fig:ICU-fit} and \ref{fig:Death-fit}.

\section{Study and design of testing policies} \label{sec:control}

In this section, we use the model validated with the COVID-19 data of France to design two types of testing policies. The first one assumes that the total stockpile of tests is unlimited, but the testing capacity per day is limited. This results in the so called \textit{Best Effort Strategy for Testing (BEST)}, which gives the minimum number of tests needed to be performed per day in order to stop the epidemic from growing. In other words, if the BEST is applied, then the number of new infections stop to grow with respect to time. 

The second strategy, named COST (\textit{Constant Optimal Strategy for Testing}) assumes that the total stockpile of tests in a country is limited. In such a case, we investigate the optimal testing policy per day that results in a minimum epidemic peak. In contrast with BEST, there is an optimal value of tests to be performed per day which is smaller than the maximum testing capacity per day.

\subsection{Best effort strategy for testing}

Assume that the total stockpile of tests is unlimited during the whole epidemic period. This assumption is valid for a country that can manufacture or buy tests continuously during the time of the epidemic. Based on this, we provide a testing policy recommendation on the daily testing capacity starting from a certain time $t^*$ in order to change the course of the epidemic in a sense that is defined below. For simplicity, we further assume that the number of tests performed per day is considered to be the daily testing capacity. In other words, the daily testing capacity is utilized completely each day, i.e., $u(t) = c(t)$.

We say that, at any time $t$, an epidemic is {\em spreading} if the number of undiagnosed infected population $x_\I(t)$ is increasing, i.e., the effective reproduction number $R_t>1$. On the other hand, an epidemic is {\em non-spreading} if $x_\I(t)$ is not increasing, i.e., the effective reproduction number $R_t \leq 1$.

\begin{Definition}[BEST]
The {\it best effort strategy for testing (BEST) at a given time $t^*$} is the minimum number of tests to be performed per day from time $t^*$ onward such that the epidemic switches from spreading to non-spreading at $t^*$.
\end{Definition}

In other words, BEST provides the smallest lower bound on the number of tests performed per day, sufficient to change at given time $t^*$ the course of the epidemic from spreading to non-spreading.
In order to compute the BEST at a given time $t^*$, we first define the following function:
\begin{equation}\label{c_star} 
c^*(t)=x_\T(t)\left|\frac{ \beta(t)}{N}x_\S(t)-\gamma\right|_+
\end{equation} 
where, by definition, for any scalar $z$, $|z|_+=z$ if $z > 0$ and $|z|_+=0$ otherwise. 
\begin{Proposition} \label{prop:BEST}
Assume that the infection rate $\beta$ is non-increasing while the testing specificity parameter $\theta$ is non-decreasing on a time interval $[t^*,t_1)$, for some $t^* < t_1$. Then, the best effort strategy for testing (BEST) at time $t^*$ is given by 
\[
u(t)=c^*(t^*)=x_\T(t^*)\left|\frac{ \beta(t^*)}{N}x_\S(t^*)-\gamma\right|_+,\quad \forall t\in[t^*,t_1).
\] 
\end{Proposition}
\begin{proof}
\mbox{}
In order to prove that $c^*(t^*)$ for $t\in [t^*,t_1)$ is the BEST at $t^*$,  we show the following:
\begin{enumerate}[(i)]
    \item If $u(t) > c^*(t)$ (resp., $u(t) \geq c^*(t)$) for any $t\in [t^*,t_1)$, then $x_\I$ is decreasing (resp., non-increasing) on $[t^*,t_1)$.
    \item If $u(t) > c^*(t^*)$ (resp., $u(t) \geq c^*(t^*)$) for any $t\in [t^*,t_1)$, then $x_\I$ is decreasing (resp., non-increasing) on $[t^*,t_1)$.
\end{enumerate}

Assume that $u(t) > c^*(t)$ on $[t^*,t_1)$.
Then, 
$
\Phi(t) := \beta(t) \frac{x_\S(t)}{N} -\frac{u(t)}{x_\T(t)} - \gamma < 0
$
which implies that $x_\I$ is decreasing since $\dot x_\I(t) = \Phi(t) x_\I(t)$ almost everywhere. 
If only the weaker assumption $u(t) \geq c^*(t)$ on $[t^*,t_1)$ is fulfilled, then, by using the continuity of the solutions of ODE with respect to perturbations of the right-hand side, one gets that $x_\I$ is non-increasing.

Assume now that $u(t) > c^*(t^*)$ on $[t^*,t_1)$, where $c^*(t^*)$ is constant. Then, by continuity, $u(t) > c^*(t)$ on a certain interval $[t^*,t_2)$, for some $t_2 \in (t^*,t_1]$. As a consequence of the result (i) shown previously, $x_\I$ decreases on $[t^*,t_2)$. Moreover, assume that $t_2$ is the maximal point in $(t^*,t_1]$ having this property. In order to show that $t_2=t_1$, it is sufficient to show that $u(t_2)> c^*(t_2)$, otherwise one may consider a larger value for $t_2$ which will lead to a contradiction with the fact that it is maximal. Since $x_\I$ decreases on $[t^*,t_2)$ and $\theta$ is non decreasing, from Lemma~\ref{lemma_x_T_decreasing} (see \ref{sec_lemma}), we can conclude that $x_\T$ also decreases on this interval. On the other hand, since $x_\S$ is always decreasing and $\beta$ is non increasing, one can conclude that $c^*(t)$ also decreases on $[t^*,t_2)$. This is obtained by upper bounding $c^*(t)$. Thus, one has $c^*(t^*) > c^*(t)$, which implies that that $u(t_2)>c^*(t_2)$.
Therefore, as $t_2=t_1$, we have established that $x_\I$ decreases on the whole interval $[t^*,t_1)$.
For the case where $u(t) \geq c^*(t^*)$, we can use the same argument of continuity of the trajectories.

From the previous results, one deduces that the BEST is given by $c^*(t^*)$ and the testing rate $u(t) \geq c^*(t^*)$ for $t\in[t^*,t_1)$.
If $u(t)<c^*(t^*)$, for $t\in[t^*,t_1)$, then one can show easily that the epidemic goes on spreading in the interval $[t^*,t_1)$. Hence, $u(t)=c^*(t^*)$ is the BEST policy at $t^*$. 
\end{proof}

\begin{algorithm}[ht!]
\begin{enumerate}
    \item Inputs: $N$, $\beta$, $\gamma$, $\theta$, $t^*$, $x_\S(t^*)$ and $x_\T(t^*)$.
    \item \label{step_return} Compute the BEST policy $c^*(t^*)$ using \eqref{c_star}.
    \item Set $u(t)=c^*(t^*)$, for all $t\geq t^*$.
    \item Return to step \ref{step_return} if $\beta$ increases or $\theta$ decreases. 
\end{enumerate}
\caption{Computation of the BEST policy at time $t^*$.}
\label{algo:peak}
\end{algorithm}

Proposition~\ref{prop:BEST} states that the peak of $x_\I(t)$ is uniquely determined by the BEST policy $c^*(t^*)$, where the peak is achieved at time $t^*$. Therefore, Algorithm~\ref{algo:peak} can be used to set the peak time $t^*$ once parameters $\beta$, $\gamma$, and $\theta$ are learned from the data.

\begin{Remark}
Requiring that $\beta$ must not increase and $\theta$ must not decrease in the interval $(t^*,t_1)$ for some $t_1>t^*$ is necessary for the BEST policy. 
It is thus important to keep the external conditions that determine the values of $\beta$ and $\theta$ either constant or such that $\beta$ decreases (e.g., through the implementation of lockdown) and/or $\theta$ increases (e.g., through efficient contact tracing).
\end{Remark}

\begin{Remark}
The case where $\beta$ decreases and/or $\theta$ increases at some time $t_1>t^*$ has the effect of speeding up the suppression of the epidemic under BEST policy.
\end{Remark}

\begin{Remark}
From \eqref{eq401b}, we can note that if $x_\S(t)/N < \gamma / \beta$, the the epidemic naturally decreases. In this case, doing no testing $u(t)=0$ is the BEST policy, which, by definition, gives a minimum number of tests to be performed in order to stop the growth of the infected population $x_\I$. However, if testing is resumed in this case, i.e., $u(t)>0$, it will further speed up the decrease of the infected population.
\end{Remark}

\subsubsection{Evaluation of the BEST policy}
Giving data from France, we first compute $c^*(t^*)$ for different values of $t^*$ from January~24 to March~13. Figure~\ref{fig:with-BEST-2} shows the number of tests per day required by the BEST policy if it is implemented on day $k$ and the corresponding value of peak of infected cases $x_\I(k^*)$. One can note that the later BEST is applied the higher is the required number of tests. An exponential increase can even be observed from February 28 which corresponds to an acceleration of the infection.  
\begin{figure}[!htb]
    \centering
    \includegraphics[width=0.9\textwidth]{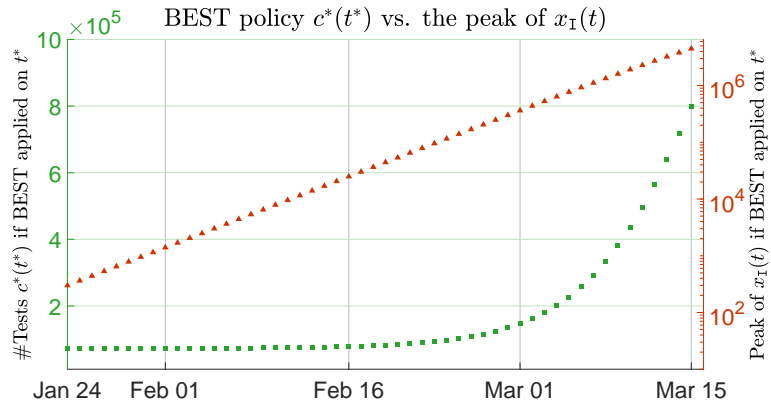}
    \caption{Number of tests per day required by the BEST policy (left y-axis, green) vs. peak of infection (right y-axis, red, in logscale) for an implementation day $t^*$.}
    \label{fig:with-BEST-2}
\end{figure}
Now, we consider a scenario where BEST is implemented on March~01. Figure~\ref{fig:with-BEST} depict the number of active cases when $u(t)$ is the actual testing scenario (see Figure~\ref{fig:input_plot}) and when $u(t)$ is given by BEST. To evaluate BEST, we use $u(t)$ as given by recorded data before March~01 then use $u(t)$ given by BEST from March~01. In the first case, the peak of the infected population $x_\I(t)$, which are the active undiagnosed cases, is about $6$ million. In the second case, the peak of infected population in this case is $363,169$. The required number of tests per day to be performed for the implementation of BEST on March~01 is $c^*\approx 147,000$.

\begin{figure}[!htb]
    \centering
    \includegraphics[width=0.9\textwidth]{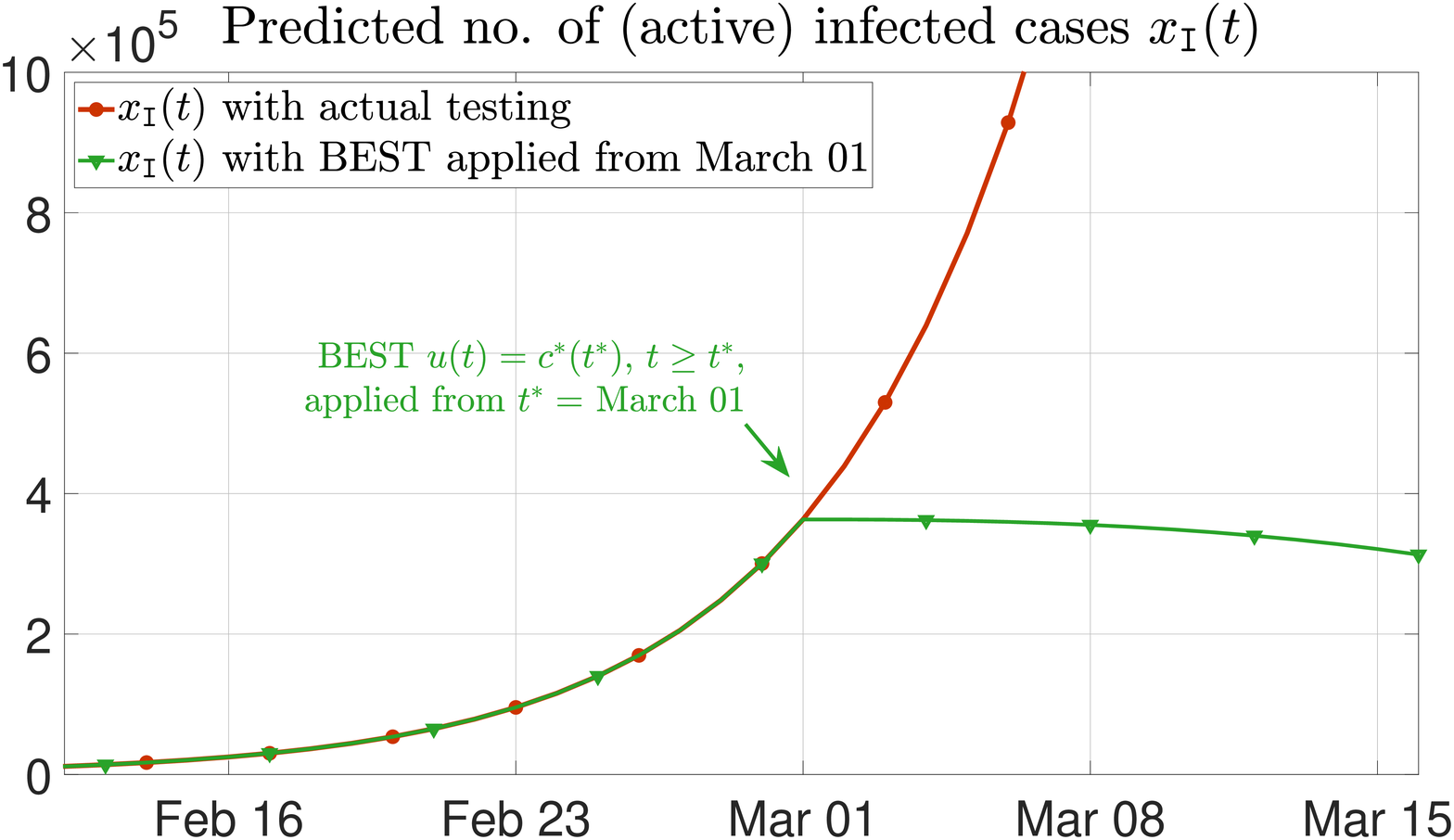}
    \caption{Predicted number of infected cases $x_\I(t)$: actual testing scenario vs. BEST}
    \label{fig:with-BEST}
\end{figure}


The impact in terms of ICU occupation and number of deaths is now evaluated using the equations \eqref{eq:icu-infection} and \eqref{eq:death-infection} respectively. 
The results are illustrated in Figure~\ref{fig:ICU-BEST} and \ref{fig:Death-BEST}. 
\begin{figure}[!htb]
    \centering
    \includegraphics[width=0.9\textwidth]{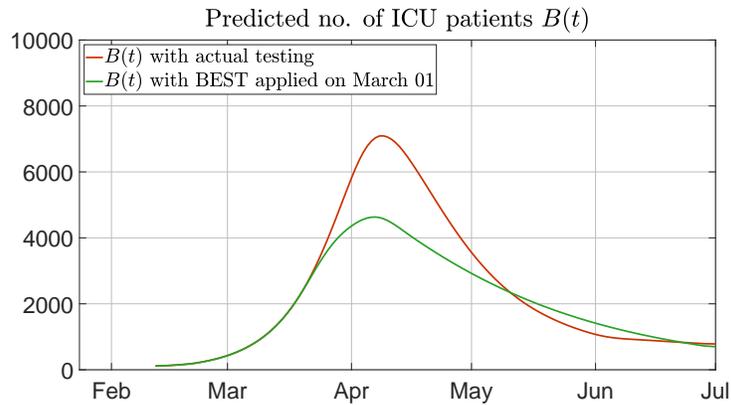}
    \caption{The prediction of the number of active ICU cases $B(t)$: actual scenario vs. BEST policy.}
    \label{fig:ICU-BEST}
\end{figure}
\begin{figure}[!htb]
    \centering
    \includegraphics[width=0.9\textwidth]{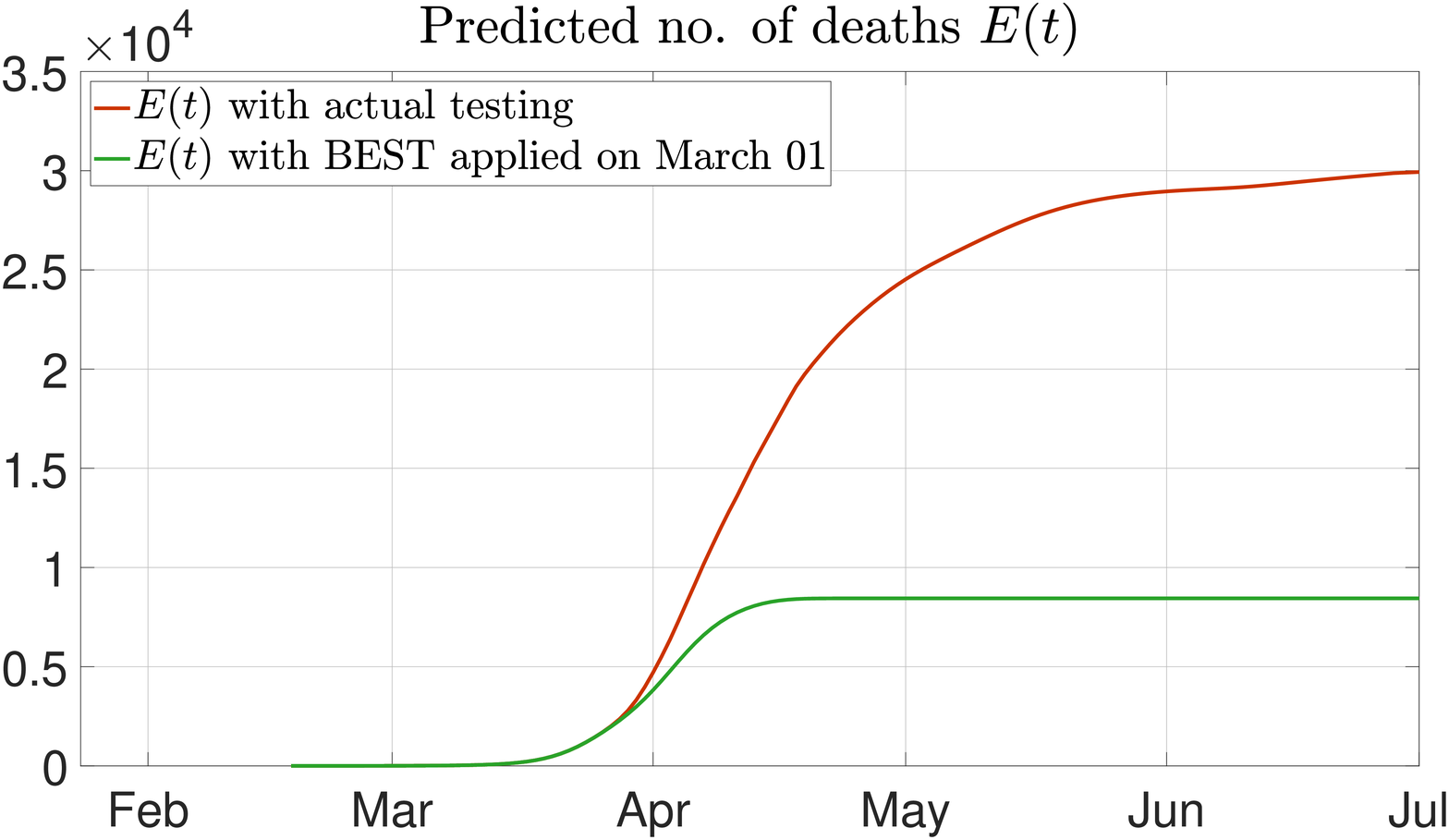}
    \caption{The prediction of the cumulative number of deaths $E(t)$: actual scenario vs. BEST policy.}
    \label{fig:Death-BEST}
\end{figure}

We observe that the peak of the number of active ICU patients could have been reduced by $34.71\%$ and the number of deaths could have been reduced by $74.45\%$ if the BEST policy was applied from March 01, 2020.


\subsection{Constant optimal strategy for testing}
\label{sec_optimal_testing}

Consider now that the total stockpile of tests is limited and given by $r_{\max}$.
In this case, the testing rate should be chosen carefully so as to not finish the stockpile of tests too early by performing too many tests per day, or to be unable to control the epidemic spread by performing too few tests per day. In the former case, when the stockpile of tests finishes too early due to intensive testing, the infections, even if significantly reduced in the beginning, will start to spread again and result in a second wave of the epidemic with an infection peak much higher than before. On the other hand, if the number of tests performed per day is too little, this would result in the infected population reaching a very high peak during the first wave of the epidemic, which could challenge the available medical facilities of a country. Thus, in this section, we determine a constant optimal allocation of a limited stockpile of tests $r_{\max}$ such that both peaks of the infected population are minimized.

Given the total stockpile of tests $r_{\max}$,  we assume that the number of  tests performed per day is given by
\begin{equation} \label{u-COST}
    u(t) = \left\{ \begin{array}{ll}
          C,    & \text{if}~ 0 \leq t \leq  T \\
          0, & \text{if}~ t>T
    \end{array} \right.
\end{equation}
where the time period $T:=r_{\max} / C$ represent the duration of the testing policy once $C$ is determined.

\begin{Definition}[COST]
The constant optimal strategy for testing (COST) with the total stockpile of tests $r_{\max}$ is the policy of class \eqref{u-COST} that minimizes the peak of the infected population $x_\I(t)$.
\end{Definition}

In other words, the COST policy allocates the limited stockpile of tests $r_{\max}$ as $C$ tests per day for the time interval $[0,T]$, where $T$ is measured in days, such that the maximum value of the infected population $x_\I(t)$ is minimized.

In the following, we first study the SIDUR model in a new coordinate of `infection' time $\xi$ with the aim of finding analytic solutions. Then, we use those solutions to compute the two peak values of $x_\I(t)$, where the first peak arrives at $t\leq T$ and the second at $t>T$. Finally, we compute the optimal value of $C$ that minimizes those peak values by equating them, which is depicted in Figure~\ref{fig:two_peaks}.

\begin{figure}[!htb]
    \centering
    \includegraphics[width=0.9\textwidth]{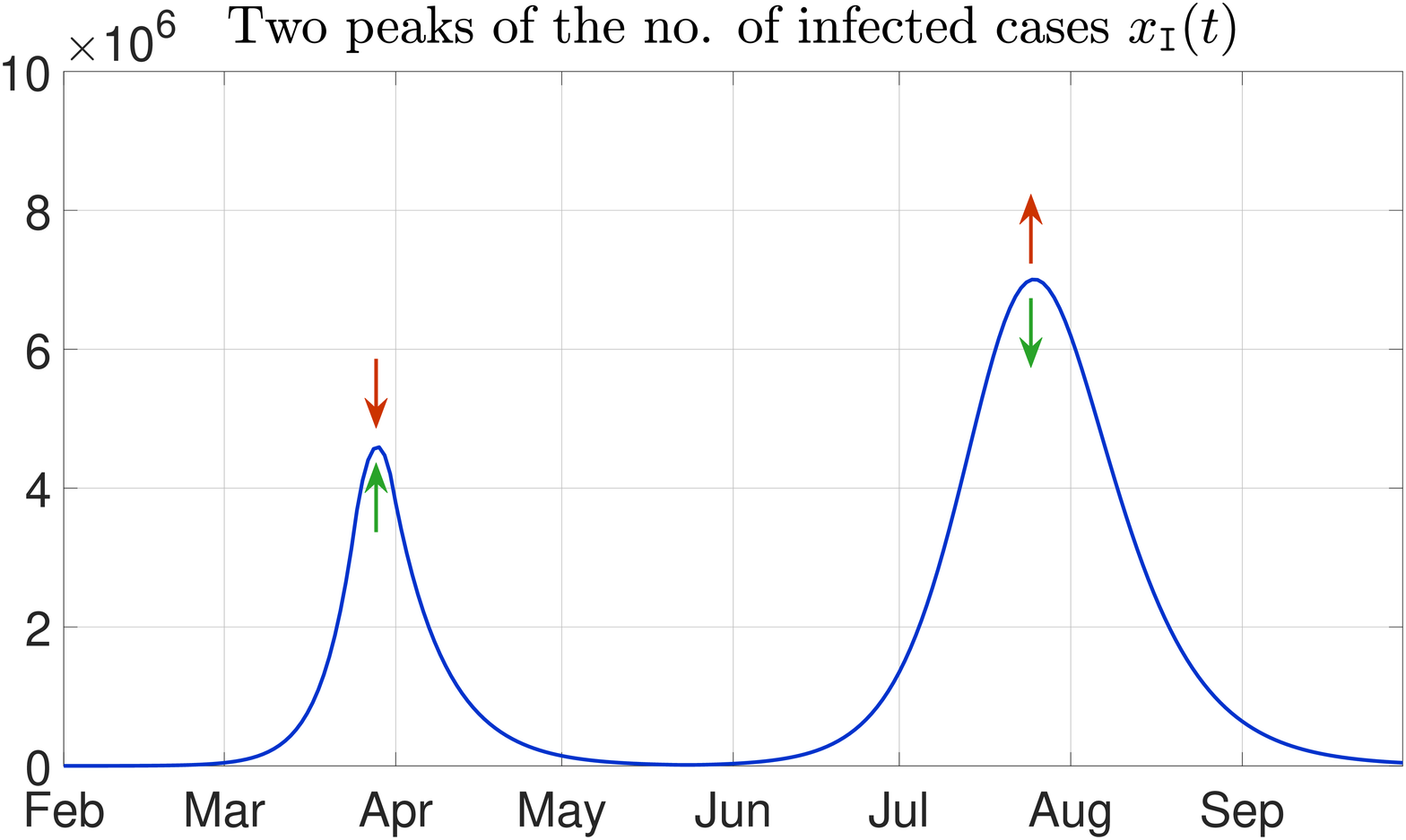}
    \caption{Under the limited stockpile of tests $r_{\max}$, the number of infected cases $x_\I(t)$ will have two peaks. The first peak happens before the stockpile of tests is finished and the second peak happens after the stockpile finishes. The second peak is the result of the second wave of the epidemic during which no testing is done. If the testing policy is chosen such that the first (respectively, second) peak is reduced, then the second (respectively, first) peak will increase. The minimum of both peaks is achieved when the peak values of both waves of the epidemic are equal. This is ensured by COST policy $C$.}
    \label{fig:two_peaks}
\end{figure}

\subsubsection{Partial solution to the SIDUR model}


The SIDUR model has an interesting property: most of the terms on the right-hand side of the model \eqref{eq401} depend linearly on $x_\I(t)$. This gives an intuition that the epidemic process goes faster when the number of infected people $x_\I$ is high.
Indeed, this idea was used  very early  to analyze the basic SIR model, see \cite{kermack1927}. It is possible to define a new time variable $\xi$, which we call {\em infection time}, as an integral of undetected infected people up to time $t$, i.e., 
\[
d\xi = x_\I(t) dt.
\]
This means exactly that the speed of the processes in the new infection time is proportional to $x_\I$.
The `real' time can be reconstructed from the infection time by:
\begin{equation} \label{time_xi}
t(\xi) = \int\limits_0^\xi \frac{d\xi'}{x_\I(\xi')}.
\end{equation}
Note that the infection timescale preserves the peak value of epidemic, which is of interest to us. That is,
$$
\frac{dx_\I}{dt} = \frac{dx_\I}{d\xi} \frac{d\xi}{dt} = \frac{dx_\I}{d\xi} x_\I
$$
and $dx_\I/dt = 0$ implies $dx_\I/d\xi = 0$ and vice versa (as long as $x_\I \neq 0$). Therefore, one can perform a model analysis based on the infection timescale $\xi$. Moreover, by writing the equations in $\xi$, the partial solution of SIDUR model can be obtained analytically. 

\begin{Proposition} Evolution of the unidentified recovered population $x_\U$ is affine with respect to the infection time $\xi$, whereas the susceptible population $x_\S$ is exponential with respect to $\xi$. That is,
\begin{eqnarray}
\label{x_R_sol}
x_\U(\xi) &=& x_\U(0) + \gamma\xi=\gamma\xi \\ \label{x_S_sol}
x_\S(\xi)&=& x_\S(0)e^{-\frac{\beta}{N}\xi}.
\end{eqnarray}
\end{Proposition}
\begin{proof}
It is straightforward to see that the solution of \eqref{eq401d} is given by \eqref{x_R_sol}. On the other hand, \eqref{eq401a} yields 
\[
\frac{dx_\S(t)}{d\xi}=-\frac{\beta}{N}x_\S(t)
\] 
whose solution is given by \eqref{x_S_sol}.
\end{proof}
It is not so trivial to find analytic solutions for $x_\I$, $x_\D$, and $x_\R$. However, under certain hypothesis on the testable population, it is possible to get a reasonable approximation of the solutions. 
\begin{Assumption} \label{ass: simplification_of_X_T}
The testable population $x_\T$ is approximated as
\begin{equation} \label{eq:testable_approx}
x_\T(t) \approx (1-\theta) N.
\end{equation}
\end{Assumption}

 This makes it possible to find an analytic solution for the infected population $x_\I(t)$.

\begin{Proposition} Evolution of the infected population $x_\I$ under Assumption~\ref{ass: simplification_of_X_T} is given by
\begin{equation}\label{x_I_sol_general}
x_\I(\xi) = x_\I(0) + x_\S(0) \left(1 - e^{-\frac{\beta}{N}\xi} \right) - \frac{1}{(1-\theta)N} \int\limits_0^\xi u(s) ds - \gamma \xi .
\end{equation}
\end{Proposition}
\begin{proof}
Equation \eqref{x_I_sol_general} can be directly obtained by integrating \eqref{eq401b}, if one first transforms it using the infection time and then substitutes the definition for $x_\T$ from Assumption~\ref{ass: simplification_of_X_T} and the solution for $x_\S$ from \eqref{x_S_sol}.
\end{proof}

Using \eqref{u-COST}, we can now write a complete evolution of the infected population:
\begin{equation}\label{x_I_sol}
x_\I(\xi) = \left\{
\begin{aligned}
&x_\I(0) + x_\S(0) \left(1 - e^{-\frac{\beta}{N}\xi} \right) - \frac{ C \xi}{(1-\theta)N} - \gamma \xi, \quad \xi < \xi^* \\
&x_\I(0) + x_\S(0) \left(1 - e^{-\frac{\beta}{N}\xi} \right) - \frac{ C \xi^*}{(1-\theta)N} - \gamma \xi, \quad \xi \ge \xi^*
\end{aligned} \right.
\end{equation}
where $\xi^*$ denotes the moment of stopping the testing in the infection time coordinates, which by definition of the infection time is given by an implicit formula
\begin{equation}\label{xi_star_def}
\int\limits_0^{\xi^*} \frac{d\xi}{x_\I(0) + x_\S(0) \left(1 - e^{-\frac{\beta}{N}\xi} \right) - \frac{ C \xi}{(1-\theta)N} - \gamma \xi} = T =  \frac{r_{\max}}{C}
\end{equation}
derived from the relation
$$
\int\limits_0^{\xi^*} \frac{d\xi}{x_\I(\xi)} = T
$$
where we substitute the solution for $x_\I(\xi)$ given by \eqref{x_I_sol_general}, and the fact that $u(t) = C$ during the testing interval $[0, T]$.

\subsubsection{Analysis of the evolution of infection peaks}
From the definition of the effective reproductive number \eqref{eq:eff_R}, for convenience, we can introduce two reproduction numbers at the origin: $R_C$ with the COST policy $C$ and $R_W$ without testing: 
\begin{equation} \label{COST_reproduction}
R_C = \frac{x_\S(0) \beta}{\frac{C}{1-\theta} + \gamma N}, \quad
R_W = \frac{x_\S(0) \beta}{\gamma N}, \quad R_W>R_C.
\end{equation}
In general the evolution of the infected population given by \eqref{x_I_sol} can have at most two peaks, one during testing and one after stopping. To have guarantees that the peaks exist, we further state the following assumption:

\begin{Assumption}  \label{ass: existence_of_peaks}
The COST policy $C$ does not suppress the epidemic at the origin, i.e., $R_C>1$.
\end{Assumption} 

To explain this assumption, suppose in contrary that the COST policy $C$ suppresses the epidemic at the origin. However, such a policy will not develop the herd immunity, and, in the case of the limited stockpile of tests $r_{\max}$, the second peak will arrive once the testing is stopped. Therefore, it is reasonable to develop at least a partial herd immunity using the controlled number of tests $C$ in the first wave of epidemic.

When the two peaks of the infected population exist, their values can be obtained by setting the derivatives of \eqref{x_I_sol} with respect to $\xi$ to zero:
\begin{equation}
\begin{aligned}
\frac{dx_{\I,peak1}}{d\xi} &= x_\S(0)\frac{\beta}{N} e^{-\frac{\beta}{N}\xi_{peak1}} - \frac{C}{(1-\theta)N} - \gamma = 0 \\
\frac{dx_{\I,peak2}}{d\xi} &= x_\S(0)\frac{\beta}{N} e^{-\frac{\beta}{N}\xi_{peak2}} - \gamma = 0
\end{aligned}
\end{equation}
where the peaks positions in the infection time $\xi$ are given as:
\begin{equation}
\label{eq888_1}
\xi_{peak1} = \frac{N}{\beta} \ln \frac{x_\S(0) \beta}{\frac{C}{1-\theta} + \gamma N}, \qquad
\xi_{peak2} = \frac{N}{\beta} \ln \frac{x_\S(0) \beta}{\gamma N}.
\end{equation}
Writing them in terms of reproduction numbers $R_C$ and $R_W$, we obtain  
\begin{equation}
\label{eq888}
\xi_{peak1} = \frac{N}{\beta} \ln R_C, \qquad
\xi_{peak2} = \frac{N}{\beta} \ln R_W.
\end{equation}
Since $R_W>R_C>1$, by Assumption \ref{ass: existence_of_peaks}, we get  $\xi_{peak2} > \xi_{peak1} > 0$. Further one should notice that the peak $\xi_{peak1}$ occurs only if $\xi_{peak1} < \xi^*$, otherwise it is ill-defined. Similarly, $\xi_{peak2} > \xi^*$.

The peaks values themselves are given by
\begin{equation}\label{x_I_two_peaks}
\begin{aligned}
x_{\I,peak1} &= x_\I(0) + x_\S(0) \left(1 - \frac{1}{R_C} \right) - x_\S(0) R_C \ln R_C, \\
x_{\I,peak2} &= x_\I(0) + x_\S(0) \left(1 - \frac{1}{R_W} \right) - x_\S(0) R_W \ln R_W - \frac{C \xi^*}{(1-\theta)N},
\end{aligned}
\end{equation}

It is obvious that $x_{\I,peak1}$ decreases as $C$ increases, because $x_{\I,peak1}$ represents a peak under still performed testing. 

Further, it is possible to prove that $C\xi^*$ decreases as $C$ increases  for all sufficiently large $C$. Indeed, \eqref{xi_star_def}
can be rewritten as 
$$
C \int\limits_0^{\xi^*} \frac{d\xi}{x_\I(\xi)} = r_{\max}.
$$
Taking the derivative with respect to $\xi^*$ and substituting again \eqref{xi_star_def} in place of the integral, we see that
$$
\frac{dC}{d\xi^*} \frac{r_{\max}}{C} + C \frac{1}{x_{\I}(\xi^*)} = 0,
$$
which is the same as 
$$
\frac{d\xi^*}{dC} = -\frac{r_{\max} x_{\I}(\xi^*)}{C^2}.
$$
Therefore,
$$
\frac{d(C\xi^*)}{dC} = \xi^* - \frac{r_{\max} x_{\I}(\xi^*)}{C} = \int\limits_0^{T} \left[ x_\I(\tau) - x_\I(T) \right] d\tau,
$$
where $\xi^*$ is rewritten by the definition of the infection time \eqref{time_xi}, while $r_{\max}/C$ is represented by the integral of 1 over the time $T = r_{\max}/C$. The result is negative unless $x_\I(T)$ is smaller than an average number of infected people all the time before, which means that $T$ is long after the peak has come. Thus for all reasonable sufficiently large $C$ such that $T$ is not too large $C$, $x_{\I,peak1}$ decreases, while $x_{\I,peak2}$ increases.

This property can be used to optimize a maximum between two peaks. Indeed, a maximum between a decreasing and an increasing function is minimized when they are equal each other. Thus, we equate two peak values in \eqref{x_I_two_peaks} and solve it with respect to $\xi^*$, obtaining an optimality condition as
\begin{equation}\label{xi_optimality}
\xi^*  = \frac{N}{\beta} \left( 1 + \ln R_C - \frac{R_C}{R_W - R_C} \ln \frac{R_W}{R_C} \right).
\end{equation}
Equation \eqref{xi_optimality} provides an optimality condition which, being combined with the connection between $\xi^*$ and $C$ given by \eqref{xi_star_def}, constitutes a complete system determining the COST $C$.

It is possible to show that the optimal strategy provides switching in between of two peaks $\xi_{peak1}$ and $\xi_{peak2}$, thus both peaks exist. Indeed, first we prove a simple lemma:
\begin{Lemma} \label{lemma:log_x}
For any $x \ge 1$
$$
x - 1 \ge \ln x \ge 1 - \frac{1}{x},
$$
and the equality is possible only if $x = 1$.
\end{Lemma}
\begin{proof}
First of all, it is obvious that for $x = 1$ all three parts of the inequality are equal to zero. Further, let us take the derivative of this inequality:
$$
1 \ge \frac{1}{x} \ge \frac{1}{x^2},
$$
which is always true for all $x \ge 1$, and it holds strictly for all $x > 1$. This concludes the proof.
\end{proof}
Now, consider the difference $\xi^* - \xi_{peak1}$:
$$
\begin{aligned}
\frac{\beta}{N}\left(\xi^* - \xi_{peak1} \right) &= 1 - \frac{R_C}{R_W - R_C} \ln \frac{R_W}{R_C} \\
&= 1 - \frac{1}{x - 1} \ln x \\ &= \frac{1}{x - 1} \left(x - 1 - \ln x\right) \ge 0,
\end{aligned}
$$
where we denote $x = R_W / R_C$, and the last inequality comes from Lemma~\ref{lemma:log_x}. Thus the first peak is well-defined. Further, writing the difference $\xi_{peak2} - \xi^*$ in the same manner and using the same definition for $x$, we see that
$$
\begin{aligned}
\frac{\beta}{N}\left(\xi_{peak2} - \xi^* \right) &= \left(\frac{R_C}{R_W - R_C} + 1\right) \ln \frac{R_W}{R_C}  - 1 \\
&= \frac{x}{x-1}\ln x - 1 \\ &= \frac{x}{x-1} \left(\ln x - 1 + \frac{1}{x}\right) \ge 0.
\end{aligned}
$$
Thus we have proven that $0 < \xi_{peak1} \le \xi^* \le \xi_{peak2}$, which means the stopping of testing happens in between of the two peaks.

\subsubsection{Computation of the COST policy}

Finally, the COST policy can be obtained by solving \eqref{xi_star_def} with respect to $C$, where \eqref{xi_optimality} are used as an upper limit of the integral. This can be formally stated as follows:

\begin{Proposition}
Consider the testing rate defined by \eqref{u-COST}. Then, under Assumptions~\ref{ass: simplification_of_X_T}~and~\ref{ass: existence_of_peaks}, the COST policy $C$ is given by solving \eqref{xi_star_def}--\eqref{xi_optimality}, i.e.,
\begin{eqnarray}
\label{eq_COST_solution_1}
 0 & =& \frac{r_{\max}}{C}  -\int\limits_0^{\xi^*} \frac{d\xi}{x_\I(0) + x_\S(0) \left(1 - e^{-\frac{\beta}{N}\xi} \right) - \frac{C \xi}{(1-\theta)N} - \gamma \xi} \\
 \label{eq_COST_solution_2}
\xi^*  &=& \frac{N}{\beta} \left(1 + \ln \frac{x_\S(0) \beta}{\frac{C}{1-\theta} + \gamma N} \right) -  \frac{(1-\theta) \gamma N^2}{C \beta} \ln \frac{\frac{C}{1-\theta} + \gamma N}{\gamma N}.
\end{eqnarray}
\end{Proposition} 
The time duration for positive testing rate \eqref{u-COST} is given by $T = r_{\max} / C$. Note that, by construction, a solution to the system of equations \eqref{eq_COST_solution_1} and \eqref{eq_COST_solution_2} exists and is unique as long as Assumption~\ref{ass: existence_of_peaks} holds. This system of equations can be solved numerically for $C$ by using the Newton's method (if the initial estimate $C_0$ is chosen sufficiently close to the real optimal value $C$). The process is sketched in Algorithm~\ref{algo:COST}. 

\begin{algorithm}
\begin{enumerate}
    \item Take the initial estimate for COST $C_0$.
    \item On $n$-th iteration step, given $C_n$, find $\xi^*_n$ by \eqref{eq_COST_solution_2}.
    \item Find the value of the right-hand side of \eqref{eq_COST_solution_1}. This can be done by computing the integral numerically. Denote the result as $f(C_n)$.
    \item Find the derivative $f'(C_n)$ of the right-hand side of \eqref{eq_COST_solution_1}. This can be done analytically:
    $$
    f'(C_n) = -\frac{r_{\max}}{C^2_n} - \frac{1}{x_\I(\xi^*_n)} \frac{d\xi^*_n}{dC_n}
    $$
    where $x_\I(\xi^*_n)$ is computed from \eqref{x_I_sol} and, from \eqref{eq_COST_solution_2},
    $$
    \frac{d\xi^*_n}{dC_n} = -\frac{N}{\beta C_n} + \frac{(1-\theta) \gamma N^2}{C^2_n \beta} \ln \frac{\frac{C_n}{1-\theta} + \gamma N}{\gamma N}.
    $$
    \item Once $f(C_n)$ and $f'(C_n)$ are computed, update the COST by the Newton's method:
    $$
    C_{n+1} = C_n - f(C_n) / f'(C_n).
    $$
    \item Repeat, from step 2, until the desired accuracy is achieved.
\end{enumerate}
\caption{Computation of the COST policy.}
\label{algo:COST}
\end{algorithm}

\begin{table}[!htb]
    \centering
    \begin{tabular}{|l||l|}
        \hline
        Infection rate & $\beta_1 = 0.2643 \quad \beta_2 = 0.0006 \quad \beta_3 = 0.0642$  \\
        \hline
        Testing specificity & $\theta_1 = 0.9415 \quad \theta_2 = 0.7993$  \\
        \hline
        Recovery rate & $\gamma = 0.0542$   \\
        \hline
        Removal rate & $\rho = 0.0499$ \\ 
        \hline
    \end{tabular}
    \caption{Estimated parameter values under Assumption~\ref{ass: simplification_of_X_T}.}
    \label{tab:estimated_parameters-COST}
\end{table}

\subsubsection{Evaluation of COST policy}
We first validate Assumption \ref{ass: simplification_of_X_T}. This assumption leads to a new set of model parameters given by Table~\ref{tab:estimated_parameters-COST}.
With these parameters, the model fitting is  illustrated in Figure~\ref{fig:validation-COST}.  
\begin{figure}[!htb]
    \centering
    \includegraphics[width=0.9\textwidth]{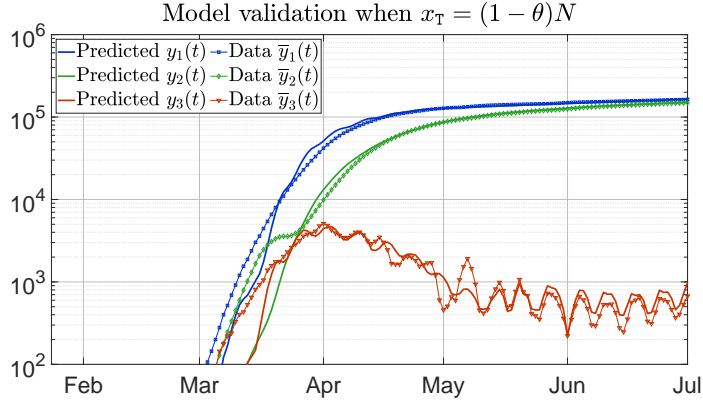}
    \caption{Model fitting under Assumption~\ref{ass: simplification_of_X_T} and parameter values of Table~\ref{tab:estimated_parameters-COST}.}
    \label{fig:validation-COST}
\end{figure}

Similar to BEST, we evaluate the COST through the prediction of the number of active ICU cases $B(t)$ and the cumulative number of deaths $E(t)$ using equations \eqref{eq:icu-infection} and \eqref{eq:death-infection}, respectively. However, we fit these equations to the data $\ol{B}(k)$ and $\ol{E}(k)$ by using the model parameters of Table~\ref{tab:estimated_parameters-COST} that are estimated under the Assumption~\ref{ass: simplification_of_X_T}.

\begin{figure}[!htb]
    \centering
    \includegraphics[width=0.9\textwidth]{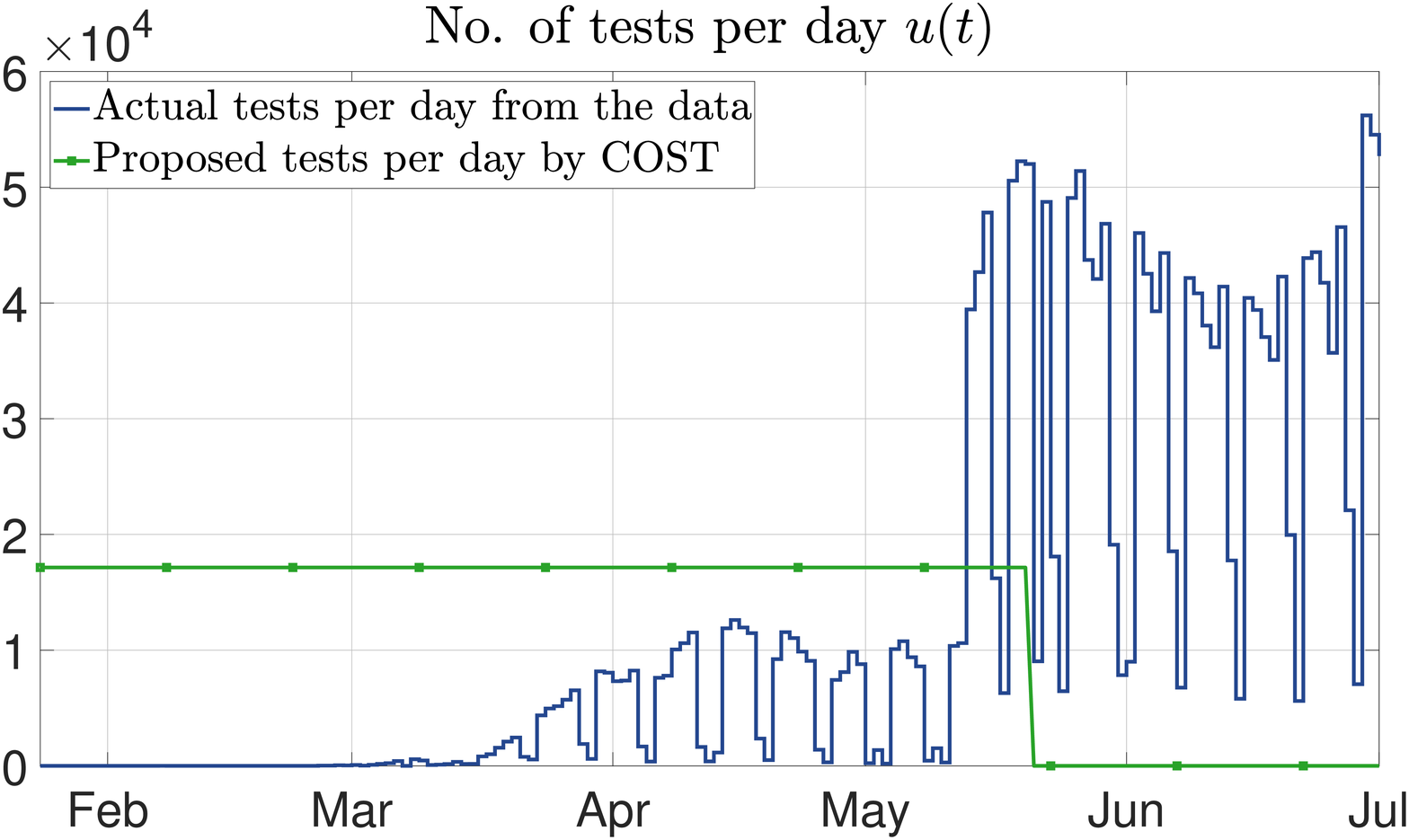}
    \caption{The comparison between the actual control input $u(t)$ versus the proposed control input $u(t)$ required by COST.}
    \label{fig:Input-COST}
\end{figure}

\begin{figure}[!htb]
    \centering
    \includegraphics[width=0.9\textwidth]{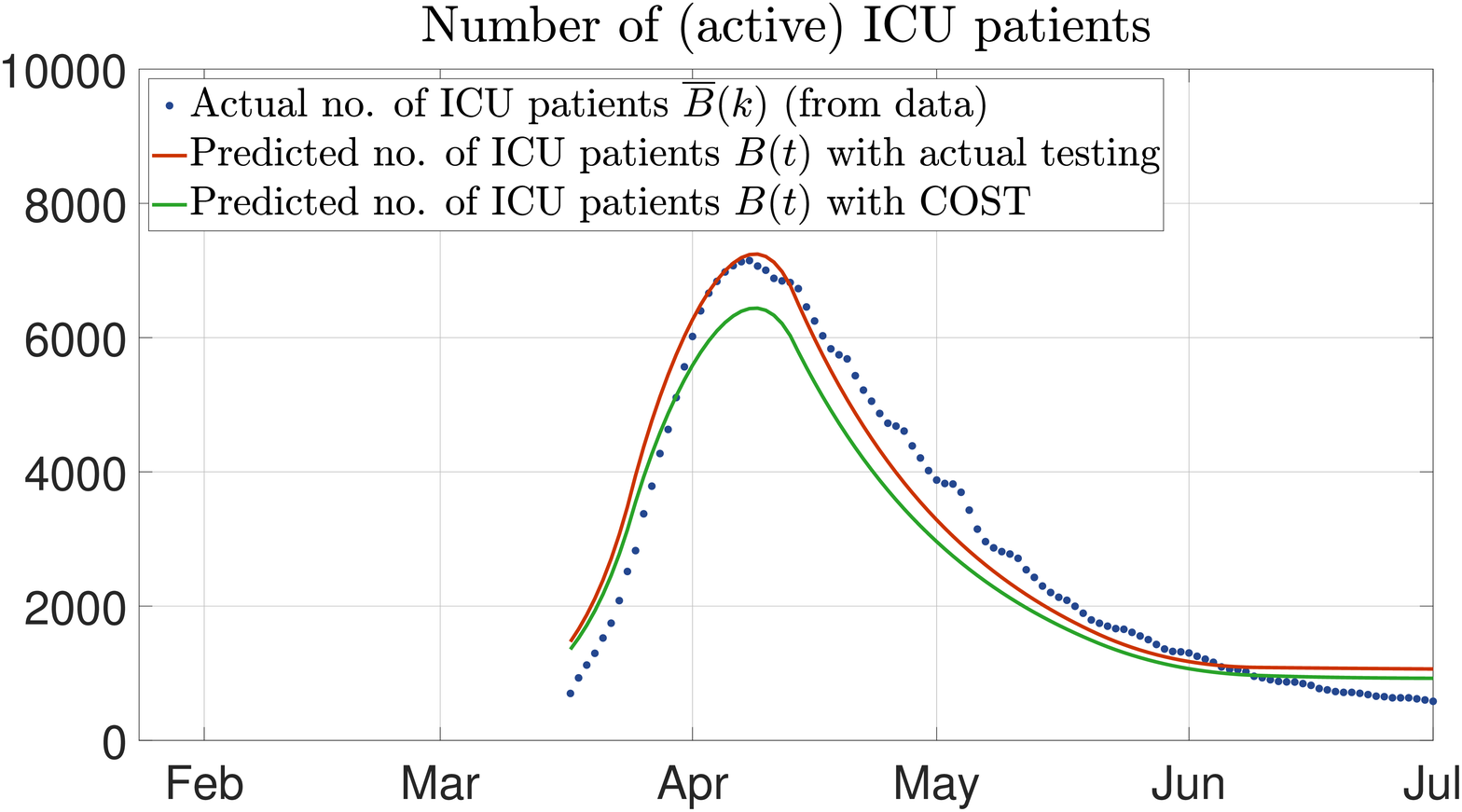}
    \caption{The comparison between the predicted number of active ICU cases $B(t)$ in the actual testing scenario and with COST policy.}
    \label{fig:ICU-COST}
\end{figure}
\begin{figure}[!htb]
    \centering
    \includegraphics[width=0.9\textwidth]{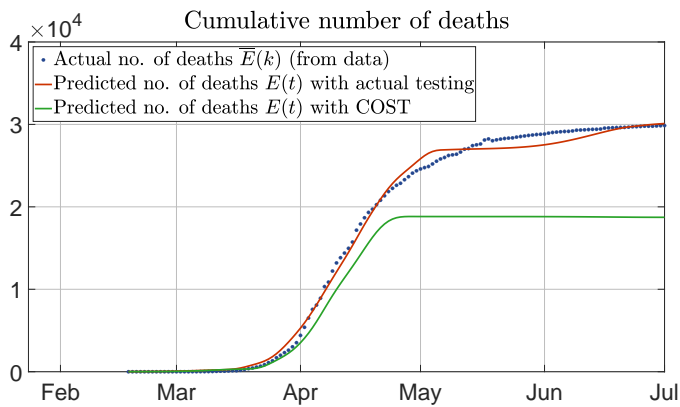}
    \caption{The comparison between the predicted cumulative number of deaths $E(t)$ in the actual testing scenario and with COST policy.}
    \label{fig:Death-COST}
\end{figure}

Assuming the total stockpile of tests $r_{\max}=2,038,037$, which is the total number of tests performed from January 24 to July 01, 2020, in France, we obtain the optimal number of tests to be performed per day as $C=17,144$ from Algorithm~\ref{algo:COST}, where the terminal time $T=118$ (i.e., May 20). This is illustrated in Figure~\ref{fig:Input-COST} along with the actual number of tests per day from the data.

Using the actual number of tests per day, we fit the model outputs \eqref{eq:icu-infection} and \eqref{eq:death-infection} with the data of the number of active ICU cases $\ol{B}(k)$ and the cumulative number of deaths $\ol{E}(k)$, where the parameter values are given in \ref{appendix_icu-death}, Table~\ref{tab:param_death-ICU_COST}. Then, we predict the number of active ICU cases $B(t)$ and the cumulative number of deaths when the COST is applied instead of actual number of tests, as illustrated in Figure~\ref{fig:ICU-COST} and Figure~\ref{fig:Death-COST}, respectively. We observe that COST reduces the peak of $B(t)$ by $11.12\%$ and the total number of cumulative deaths $E(t)$ by $37.52\%$. 

Notice that the total stockpile of tests in the actual testing and COST is the same. However, most of the tests in the actual testing are consumed after May 13, whereas all the tests in COST are allocated equally in the time interval January 24 to May 20. The number of tests performed per day in the actual testing exceeds $50,000$ per day after May 13, which is about three times more than what is required by COST, i.e., $17,144$. Thus, COST is optimal in a sense that it is practical, it decreases the burden on medical facilities, and it reduces the number of deaths significantly.

\section{Concluding remarks}

We proposed a SIDUR model for the control of the COVID-19 epidemic through the number of RT-PCT tests performed per day. Such tests enable the government to diagnose and isolate the infected people from the susceptible population. We estimated and validated the model on the French \mbox{COVID-19} data, and proposed two testing policies to control the epidemic: 1) best effort strategy for testing (BEST) and 2) constant optimal strategy for testing (COST). BEST provides the minimum number of tests to be performed from a certain day onward in order to make the increasing infected population non-increasing immediately. That is, it changes the course of epidemic from spreading to non-spreading. On the other hand, COST considers a limited stockpile of tests that are optimally allocated in a time interval starting from the beginning of the epidemic in order to minimize the peak of infected population.

    The control input in SIDUR model corresponds to the number of RT-PCR tests performed per day. However, another type of test, a serology test, which is not considered in the current model because of the unavailability of its data, is also very important. A serology test determines the relevant antibodies in a subject's serum in order to detect whether he/she was infected in the past. By performing serology tests on the testable population, one can detect the unidentified recovered population and transfer them in the identified removed compartment of the model. This reduces the size of the testable population, which in turn increases the testing specificity of RT-PCR tests. In other words, the serology tests complement the RT-PCR tests \cite{walque2020,winter2020}. Therefore, as a future prospect, it will be interesting to consider two control inputs corresponding to both types of test in the SIDUR model.
    
    The model is estimated and validated by fitting the model outputs with the available data of France. This allows us to predict the unmeasured states of the model. However, there is no certainty whether the predicted states correspond to the reality. Another prospect is to design an observer for the SIDUR model in order to estimate the true states of the model.
    
    Both BEST and COST policies are easy to compute and implement, however they are static. Thus, their influence on the control of epidemic is limited. In future, it will be interesting to solve a finite/infinite-horizon optimal control problem to minimize the peak and/or cumulative number of the infected population by a dynamic control input.

\section*{Acknowledgment}
This work is partially supported by European Research Council (ERC) under the European Union’s Horizon 2020 research and innovation programme (ERCAdG no. 694209, Scale-FreeBack, website: \url{http://scale-freeback.eu/}) and by Inria, France, in the framework of the COVID-19 fast-track research program.

\section*{Data and code availability}
The raw data used in this paper is available at \url{https://www.data.gouv.fr/fr/datasets/} with the following pointers: \href{https://www.data.gouv.fr/fr/datasets/chiffres-cles-concernant-lepidemie-de-covid19-en-france/}{Chiffres-cl\`{e}s concernant l'\'{e}pid\'{e}mie de COVID-19 en France}, \href{https://www.data.gouv.fr/fr/datasets/donnees-relatives-aux-tests-de-depistage-de-covid-19-realises-en-laboratoire-de-ville/}{Donn\'{e}es relatives aux tests de d\'{e}pistage de COVID-19 r\'{e}alis\'{e}s en laboratoire de ville}, and \href{https://www.data.gouv.fr/fr/datasets/donnees-relatives-aux-resultats-des-tests-virologiques-covid-19/}{Donn\'{e}es relatives aux r\'{e}sultats des tests virologiques COVID-19 (SI-DEP)}. The imputed data and codes are available at \url{http://scale-freeback.eu/}.


\appendix

\section{Particle swarm optimization} \label{appendix_PSO}
In this appendix, we briefly describe the particle swarm optimization (PSO) algorithm, \cite{kennedy1995}, which is considered to be one of the most powerful algorithms. It considers a foraging swarm of $n$ particles who collectively search for an optimal solution of \eqref{prob:param_est} in the parameter space. At time step $h=0,1,2,\dots$, each particle~$i$ visits a position $\hat{p}^i_h$ by moving with velocity $v^i_h$. Initially, when $h=0$, the positions $\hat{p}^i_0$, for all $i\in\{1,\dots,n\}$, are chosen randomly in the parameter space and the velocities $v^i_0=0$. Each particle~$i$ stores its personal best pair $(\hat{p}^{i*}_h,J^{i*}_h)$ and the social best pair $(s^{*}_h,J^{s*}_h)$ in memory, where $J^{i*}_h = \mc{J}(\hat{p}^{i*}_h)$ and $J^{s*}_h = \mc{J}(s^{*}_h)$ are the costs \eqref{eq:cost_fun} of personal best position $\hat{p}^{i*}_h$ and social best position $s^*_h=\arg\min_{\hat{p}^{i*}_h, i\in\{1,\dots,n\}} \mc{J}(\hat{p}^{i*}_h)$, respectively. Notice that $J^{s*}_h\leq J^{i*}_h$ for all $i\in\{1,\dots,n\}$. The personal best pair of a particle corresponds to the best position in the parameter space it has visited so far. The social best pair, on the other hand, corresponds to the best position in the parameter space that anyone in the swarm has visited so far. 

At every time step, each particle updates its velocity, position, its personal best pair, and the social best pair. The velocity and position are updated as follows:
\be \label{PSO_eq}
\ba{ccl}
v^i_{h+1} &=& w v^i_h + c_1 r_{h,1} (\hat{p}^{i*}_h - \hat{p}^{i}_h)  + c_2 r_{h,2} (s^*_h - \hat{p}^{i}_h ) \\ [0.5em]
\hat{p}^i_{h+1} &=& \hat{p}^i_h + v^i_{h+1}
\ea
\ee
where $w$ is the inertia weight, $c_1,c_2$ are the acceleration coefficients, and $r_{h,1},r_{h,2}$ are uniformly distributed random numbers in $[0,1]$ generated at each time step $h$. There are many ways of choosing these parameters \cite{clerc2002,poli2007,zhan2009}.

Each particle~$i$ computes the cost $J^i_{h+1}=\mc{J}(\hat{p}^i_{h+1})$ at its current position and updates its personal best pair as
\be \label{PSO_personal}
(\hat{p}^{i*}_{h+1},J^{i*}_{h+1}) = \left\{\ba{ll}
(\hat{p}^i_{h+1},J^i_{h+1}), & \text{if } J^i_{h+1} \leq J^{i*}_h \\ [0.5em]
(\hat{p}^{i*}_h,J^{i*}_h), & \text{otherwise}.
\ea\right.
\ee
Each particle~$i$ then communicates its personal best pair with all the other particles and each of them finds the social best pair for time $h+1$ as
\[
(s_{h+1},J^s_{h+1}) = (\hat{p}^{b}_{h+1},J^{b}_{h+1})
\]
where $b = \arg\min_{j\in\{1,\dots,n\}} J^j_{h+1}$. Finally, the social best pair is updated as
\be \label{PSO_social}
(s^*_{h+1},J^{s*}_{h+1}) = \left\{\ba{ll}
(s_{h+1},J^s_{h+1}), & \text{if } J^s_{h+1} \leq J^{s*}_h \\ [0.5em]
(s^{*}_h,J^{s*}_h), & \text{otherwise}.
\ea\right.
\ee

\section{Decreasing property of the testable population}\label{sec_lemma}
\begin{Lemma}\label{lemma_x_T_decreasing}
The testable population $x_\T$ decreases on any interval  on which $x_\I$ is decreasing and $\theta$ is non-decreasing.
\end{Lemma}
\begin{proof} Let us consider an interval $(t,t')$, $t<t'$ on which $x_\I$
is decreasing while $\theta$ is non-decreasing. First, one can note that 
\begin{eqnarray}
x_\T(t)
& = &
\nonumber
\theta(t) (x_\I(t)+x_\D(t)+x_\R(t)-N) + (N- x_\D(t) - x_\R(t))\\
& \geq &
\nonumber
\theta(t') (x_\I(t)+x_\D(t)+x_\R(t)-N) + (N- x_\D(t) - x_\R(t))\\
& = &
\label{eq124}
\theta(t') x_\I(t) + (1-\theta(t')) (N- x_\D(t) - x_\R(t)),
\end{eqnarray}
because $x_\I(t)+x_\D(t)+x_\R(t)-N<0$ and $\theta$ is non-negative and non-decreasing by assumption.
Since $x_\I$ is supposed to be decreasing, then
$$
\theta(t') x_\I(t) \geq \theta(t') x_\I(t').
$$
On the other hand
$$
\dot N-\dot x_\D-\dot x_\R = -u\frac{x_\I}{x_\T} <0,
$$
meaning that $N- x_\D - x_\R$ is a decreasing function.
As $1-\theta(t')\geq 0$, one gets
$$
(1-\theta(t')) (N- x_\D(t) - x_\R(t))
\geq (1-\theta(t')) (N- x_\D(t') - x_\R(t')).
$$
Adding the two inequalities, one deduces from \eqref{eq124} that $x_\T(t) \geq x_\T(t')$ whenever $x_\I$ is decreasing on $(t,t')$.
A tighter examination shows that, as both expressions $\theta(t')$ and $1-\theta(t')$ cannot be zero together, at least one of the two terms of the sum indeed decreases between $t$ and $t'$.
Therefore, $x_\T(t)>x_\T(t')$.
\end{proof}

\section{Parameter values for curve fitting of the number of ICU cases and deaths} \label{appendix_icu-death}

\begin{table}[h!]
    \centering
    \begin{tabular}{|l||l|l|}
        \hline
        Parameters of $B(t)$ &  $b_1=-0.54\times 10^4$ & $b_2=1.25\times 10^4$ \\
        \hline
        \hline
        \multirow{5}{*}{Parameters of $E(t)$} 
        & $e_1=4.14\times 10^4$ & $e_2=7.92\times 10^5$ \\
        & $e_3=-1.27\times 10^7$ & $e_4=9.04\times 10^7$ \\
        & $e_5=-3.63\times 10^8$ & $e_6=8.81\times 10^8$ \\
        & $e_7=-1.32\times 10^9$ & $e_8=1.19\times 10^9$ \\
        & $e_9=-5.93\times 10^8$ & $e_{10}=1.25\times 10^8$ \\
        \hline
    \end{tabular}
    \caption{Estimated parameters $b_1$ and $b_2$ in \eqref{eq:icu-infection} and $e_i$, for $i=1,\dots,10$, in \eqref{eq:death-infection} when the testable population $x_\T$ is given by \eqref{eq:testable}.}
    \label{tab:param_death-ICU}
\end{table}

\begin{table}[h!]
    \centering
    \begin{tabular}{|l||l|l|}
        \hline
        Parameters of $B(t)$ &  $b_1=8.87\times 10^{-4}$ & $b_2=1.62$ \\
        \hline
        \hline
        \multirow{5}{*}{Parameters of $E(t)$} 
        & $e_1=4.03\times 10^4$ & $e_2=-1.62\times 10^6$ \\
        & $e_3=2.76\times 10^7$ & $e_4=-2.16\times 10^8$ \\
        & $e_5=9.27\times 10^8$ & $e_6=-2.34\times 10^9$ \\
        & $e_7=3.59\times 10^9$ & $e_8=-3.28\times 10^9$ \\
        & $e_9=1.63\times 10^9$ & $e_{10}=-3.44\times 10^8$ \\
        \hline
    \end{tabular}
    \caption{Estimated parameters $b_1$ and $b_2$ in \eqref{eq:icu-infection} and $e_i$, for $i=1,\dots,10$, in \eqref{eq:death-infection} when the testable population $x_\T$ is given by \eqref{eq:testable_approx}.}
    \label{tab:param_death-ICU_COST}
\end{table}


\end{document}